\documentclass[12pt]{amsart}
\usepackage{amssymb}
\usepackage{stmaryrd}

\usepackage{helvet} 

\usepackage{courier} 
\usepackage{mathrsfs} 

\reversemarginpar 
\normalmarginpar 

\let\oldmarginpar\marginpar 
\renewcommand\marginpar[1]{\-\oldmarginpar{\raggedright\small\sf #1}}

\title{Chern classes of conformal blocks}

\author{Najmuddin Fakhruddin}
 
\address{School of Mathematics, Tata Institue of Fundamental Research, Homi Bhabha Road, Mumbai 400005, India}

\email{naf@math.tifr.res.in}

\newcommand{\nc}{\newcommand}

\nc{\bs}{\backslash}
\nc{\te}{\otimes}
\nc{\lf}{\lfloor} 
\nc{\rf}{\rfloor}
\nc{\lc}{\lceil}  
\nc{\rc}{\rceil}
\nc{\lr}{\longrightarrow}
\nc{\sr}{\stackrel}
\nc{\dar}{\dashrightarrow}
\nc{\thra}{\twoheadrightarrow}

\nc{\mc}{\mathcal}

\nc{\mb}{\mathbb}
\nc{\mf}{\mathbf}
\nc{\mr}{\mathrm}
\nc{\mg}{\mathfrak}

\nc{\bP}{\mathbb{P}}
\renewcommand{\P}{\mathbb{P}}
\nc{\Q}{\mathbb{Q}}
\nc{\Z}{\mathbb{Z}}
\nc{\C}{\mathbb{C}}
\nc{\R}{\mathbb{R}}
\nc{\A}{\mathbb{A}}
\nc{\V}{\mathbb{V}}
\nc{\W}{\mathbb{W}}
\nc{\N}{\mathbb{N}}
\nc{\D}{\mathbb{D}}

\nc{\del}{\partial}

\nc{\aff}{{\A}^1}
\nc{\naive}{\!\sim_n}

\renewcommand{\l}{\lambda}
\renewcommand{\k}{\kappa}
\nc{\vt}{\vartheta}

\nc{\ovl}{\ov{\lambda}}
\nc{\vl}{\mb{V}_{\ovl}}
\nc{\dl}{\mb{D}_{\ovl}}
\nc{\mnb}{\ov{\mr{M}}_{0,n}}
\nc{\mn}{\mr{M}_{0,n}}
\nc{\mel}{\ov{\mr{M}}_{1,1}}
\nc{\mfb}{\ov{\mr{M}}_{0,4}}
\nc{\mof}{\mr{M}_{0,4}}
\nc{\mgnb}{\ov{\mr{M}}_{g,n}}
\nc{\mgn}{\ov{\mr{M}}_{g,n}}

\nc{\omc}{\ov{\mr{M}}}

\nc{\omx}{\omega_X}

\nc{\ep}{\epsilon}
\nc{\vp}{\varpi}

\nc{\wt}{\widetilde}
\nc{\wh}{\widehat}
\nc{\ov}{\overline}

\nc{\res}{\operatorname{Res}}
\nc{\pic}{\operatorname{Pic}}
\nc{\spec}{\operatorname{Spec}}

\newtheorem{thm}{Theorem}[section]
\newtheorem{prop}[thm]{Proposition}

\newtheorem{cor}[thm]{Corollary}
\newtheorem{lem}[thm]{Lemma}

\theoremstyle{definition}

\newtheorem{ques}[thm]{Question}
\newtheorem{rem}[thm]{Remark}
\newtheorem{rems}[thm]{Remarks}

\numberwithin{equation}{section}
\input{xypic}

\begin{document}

\begin{abstract}
  We derive a formula for the Chern classes of the bundles of
  conformal blocks on $\mnb$ associated to simple finite dimensional
  Lie algebras and explore its conseqences in more detail for
  $\mg{g} = \mg{sl}_2$ and for arbitrary $\mg{g}$ and level $1$. We
  also give a method for computing the first Chern class of such
  bundles on $\mgnb$ for $g>0$.
\end{abstract}

\maketitle

\section{Introduction}

The mathematical theory of conformal blocks of Tsuchiya--Kanie
\cite{tsuchiya-kanie} and Tsuchiya--Ueno--Yamada \cite{TUY} gives rise
to a family of vector bundles, parametrised by a simple Lie algebra
$\mg{g}$, a non-negative integer $\ell$ called the level, and an
$n$-tuple of dominant weights for $\mg{g}$ of ``level $\ell$'', on
$\mgnb$, the moduli stack of stable $n$-pointed curves of genus $g$.
The ranks of these bundles are given by the celebrated Verlinde
formula (see \cite{sorger-verlinde} for a survey); the purpose of this
article is to investigate the Chern classes of these bundles, with an
emphasis on $c_1$ or the determinant bundle.

These bundles of conformal blocks have been objects of interest to
algebraic geometers ever since it was realised that they can be
described in terms of sections of natural line bundles on suitable
moduli stacks of parabolic principal bundles on curves. However, our
motivation for studying these bundles is the hope that we will get
some insight into the geometry of the moduli stacks $\mgnb$
themselves, especially when $g=0$.  The reason for this is that the
bundles of conformal blocks on $\mnb$ are generated by global
sections, hence their Chern classes are all nef. In particular,
considering the first Chern classes of these bundles we get a large
collection of elements in the cone of nef divisors of $\mnb$ and our
goal is to describe the classes so obtained.

Our main result, Theorem \ref{thm:chern}, is a formula for the Chern
classes of the bundles of conformal blocks on $\mnb$. As a special
case of this we obtain explicit divisors representing the first Chern
class, Corollary \ref{cor:cherncan}, as well as a formula for its
degree on any F-curve. We refer to \S \ref{sec:results} for the
precise statements of these results which are given in terms of ranks
of auxiliary bundles of conformal blocks, which may be computed using
the Verlinde formula, and the eigenvalues of the Casimir operator of
$\mg{g}$ acting on finite dimensional representations of level
$\ell$. The essential ingredients that we use are:
\begin{itemize}
\item The KZ connection on the restriction of these bundles to $\mn$,
\item a formula for the residue of this connection along the boundary
  divisors and
\item a formula of Esnault and Verdier from \cite{EV-logderham} which
  allows one to compute the Chern classes of a vector bundle on a
  smooth projective variety in terms of the residues of a logarithmic
  connection.
\end{itemize}

For $\mg{g} = \mg{sl}_2$, we show that the non-zero determinants of
conformal blocks of level $1$ form a basis of $\pic(\mnb)_{\Q}$
(Theorem \ref{thm:basis}). For what we call the critical level, we
show that the bundles are pulled back from suitable GIT quotients
$(\P^1)^n \sslash SL_2$ (Theorem \ref{thm:crit}); in general they are
pulled back from suitable moduli spaces of weighted stable curves
constructed by Hassett (Proposition \ref{prop:hassett}).

We also specialize our results to the case of an arbitrary simple Lie
algebra $\mg{g}$ and $\ell = 1$. A particularly interesting case is
that of $\mg{sl}_m$; the critical level $\mg{sl}_2$ determinants
reappear here for suitable choices of weights and it seems likely that
all level $1$ conformal blocks have a GIT interpretation. This has
been shown to be true for symmetric weights by N.~Giansiracusa
\cite{giansiracusa-conformal}. The case of $\mg{sp}_{2\ell}$ is also
noteworthy; the bundles of conformal blocks for these Lie algebras at
level $1$ turn out to be related to $\mg{sl}_2$ conformal blocks at
level $\ell$. This gives rise to certain F-nef divisors on $\mnb$
which we do not know are nef.

For $g>0$ one may, in principal, use a method similar to the $g=0$
case for computing the Chern classes in cohomology. However, this is
more involved since the bundles of conformal blocks do not have a flat
connection --- the first Chern class of these bundles restricted to
$\mgn$ is in general non-trivial --- so we do not work out the details
here and restrict ourselves to computing the first Chern class. Basic
properties of the moduli stacks and conformal blocks reduce this to
the case of $\mel$ in which case we give an explicit formula in
Theorem \ref{thm:m11}. We see from this that for $g>0$ the first Chern
class is almost never nef. An exception, perhaps the only one, is the
case of $\mg{e}_8$ and level $1$ in which case the bundle of conformal
blocks is the fourth tensor power of the Hodge bundle. Finally, we
note that a formula for the Chern classes of the conformal blocks
bundles restricted to $\mr{M}_{g,n}$ can be obtained from the results
of Tsuchimoto \cite{tsuchimoto}.

\smallskip 

The paper is structured as follows: In \S \ref{sec:conformal} we
recall some of the properties of conformal blocks that we shall need,
mostly without proof but we give some details in cases where we do not
know precise references. In \S \ref{sec:g0} we derive our main formula
for the Chern classes in genus $0$ and deduce some corollaries. In \S
\ref{sec:sl2} we consider the case of $\mg{sl}_2$ and in \S
\ref{sec:level1} that of the general level $1$ case in more detail. We
then consider the higher genus case in \S \ref{sec:g>0} and conclude
the paper in \S \ref{sec:questions} by discussing some questions.

\subsection{Notation}
We will work over an arbitrary field $k$ of characteristic zero.
$\mg{g}$ will always denote a finite dimensional split simple Lie
algebra over $k$, $\mg{h} \subset \mg{g}$ a Cartan subalgebra, $\Delta
\subset \mg{h}^*$ the corresponding root system and
$\alpha_1,\alpha_2,\dots,\alpha_r$ a basis of the root system inducing
a partition $\Delta = \Delta^+ \cup \Delta^-$. Let $P \subset
\mg{h}^*$ be the weight lattice and $P_+ \subset P$ be the set of
dominant weights.  For $\l \in P_+$ we let $V_{\l}$ denote the
corresponding irreducible representation of $\mg{g}$.  Let $\theta \in
\Delta^+$ be the highest root and let $\mg{s} \cong \mg{sl}_2$ be the
subalgebra of $\mg{g}$ generated by $H_{\theta}$ and $H_{-\theta}$,
where $H_{\alpha}$, for a root $\alpha$, is the corresponding coroot.

We normalize the Killing form $( |)$ on $\mg{g}$ so that
$(\theta|\theta) = 2$ and use the same notation for the induced forms
on $\mg{h}$ and $\mg{h}^*$.  We let $\{X_\k\}_{\k}$ be an orthonormal
basis of $\mg{g}$ with respect to this form.  Let $h^{\vee}$ be the
dual Coxeter number of $\mg{g}$; this is half the scalar by which the
normalised Casimir element $\Omega:=\sum_{\k} X_{\k} \circ X_\k$ acts
on the adjoint representation of $\mg{g}$. For any $\l \in P_+$, we
let $c(\l)$ be the scalar by which $\Omega$ acts on $V_{\l}$; this is
equal to $(\l| \l + 2\rho)$, where $\rho$ is half the sum of the
positive roots. For $\ell$ a non-negative integer, let $P_{\ell} =
\{\l \in P_+ \ | \ (\l|\theta) \leq \ell \}$.  For $\l \in P_+$, let
$\l^*$ be the highest weight of $(V_{\l})^*$; if $\l \in P_{\ell}$
then $\l^* \in P_{\ell}$ as well.

We shall add $\mg{g}$ as a subscript or superscript to any of the
above notation, or notation introduced later, if it is necessary to
make $\mg{g}$ explicit.

\subsection{Acknowledgements}
I thank Patrick Brosnan, Sreedhar Dutta, Norbert Hoffmann, Arvind Nair
and Madhav Nori for helpful discussions.  Discussions with Valery
Alexeev about his work with Swinarksi and also the work of Hassett
were very useful and I am grateful to him for them.  I am particulary
indebted to Prakash Belkale for a long correspondence which helped me
to learn about conformal blocks and also for his comments on this
paper.

A part of this work was done while I was visiting Universit\'e
Paris-Sud (Orsay) supported by the project ``ARCUS''. I thank Laurent
Clozel for the invitation and Universit\'e Paris-Sud for its
hospitality.

\section{Conformal blocks} \label{sec:conformal} 

In this section we recall and reformulate some of the basic
definitions and results in the theory of conformal blocks that we
shall need later. The original references are \cite{tsuchiya-kanie}
and \cite{TUY} but we shall use \cite{ueno-conformal} and
\cite{looijenga-conformal} to which we shall refer for most proofs.

\subsection{Construction} \label{subsec:conf} 

Let $S$ be a smooth variety over $k$ and $\pi:\mc{C} \to S$ be a
proper flat family of curves with only ordinary double point
singularities; we do not assume that the fibres of $\pi$ are
connected. Let $\ov{p} = (p_1,\dots,p_n)$, with $p_i:S \to \mc{C}$
sections of $\pi$ whose images are disjoint and contained in the
smooth locus of $\pi$. We also assume that $\mc{C} \bs \cup_i p_i(S)$
is affine over $S$.  Let $\ovl = (\l_1,\dots,\l_n)$ be an $n$-tuple of
elements of $P_\ell$. To this data is attached a canonically defined
locally free sheaf $V_{\mc{C}}(\ov{p},\ovl)$ on $S$, called the sheaf
of covacua \cite{TUY}, \cite{ueno-conformal}, \cite{sorger-verlinde}
or, as we shall call them, conformal blocks
\cite{looijenga-conformal}.  We briefly recall the construction:

Let $\hat{\mg{g}}$ be the affine Lie algebra defined by
    \[
    \hat{\mg{g}} := \bigl ( \mg{g} \otimes k((\xi)) \bigr )\oplus k \cdot \mr{c}
    \]
    where $\mr{c}$ is in the centre of $\hat{\mg{g}}$ and the bracket is defined
    by 
    \[[X \otimes f(\xi), Y \otimes g(\xi)] = [X,Y] \otimes
    f(\xi)g(\xi) + (X,Y)\res(g(\xi)df(\xi))\cdot \mr{c} \ .
    \]
  Put
  \[
  \hat{\mg{g}}_+ = \mg{g} \otimes k[[\xi]]\xi \ , \ \ \ \hat{\mg{p}}_+ =
  \hat{\mg{g}}_+ \oplus \mg{g} \oplus k \cdot \mr{c} \ .
  \]
  For $\l \in P_{\ell}$, we extend the action of $\mg{g}$ on
  $V_{\ell}$ to an action of $\hat{\mg{p}}_+$ by setting
\begin{itemize}
\item $\mr{c}v = \ell v$ for all  $v \in V_{\l}$
\item 
$av = 0$ for all $a \in \hat{\mg{g}}_+$ and $v \in V_{\l}$.
\end{itemize}
Put
\[
\mc{M}_{\l} := U(\hat{\mg{g}})\otimes_{U(\hat{\mg{p}}_+)}V_{\l} \ .
\]
This is a  representation of $\hat{\mg{g}}$ and has a unique 
maximal proper submodule $\mc{I}_{\l}$. Then 
\[
\mc{H}_{\l} := \mc{M}_{\l} / \mc{I}_{\l}
\]
is an irreducible $\hat{\mg{g}}$ module.

\smallskip

For an integer $n>0$, let 
\[
 \hat{\mg{g}}_n := \mg{g}\otimes  \Bigl (\bigoplus_{i=1}^nk((\xi_i)) \Bigr ) \oplus k \cdot c \ .
\]
This has a natural Lie algebra structure making it into a quotient of
$(\hat{\mg{g}})^n$.

\smallskip

For $\ovl \in P_{\ell}^n$, let
\[
\mc{H}_{\ovl} := \bigotimes_{i=1}^n \mc{H}_{\l_i} \ .
\]
This is an irreducible representation of $\hat{\mg{g}}_n$; the natural
action of $(\hat{\mg{g}})^n$ factors through $\hat{\mg{g}}_n$.

We first assume that $S$ is affine, so $S = \spec(A)$ for some
$k$-algebra $A$, and we also assume that there are given isomorphisms
$\eta_i: \widehat{\mc{O}}_{\mc{C}, p_i(S)} \to A[[\xi]]$ for each $i$.
Let $B = \Gamma(\mc{C} - \cup_{i=1}^n p_i(S))$. For each $i$, using
the isomorphism $\eta_i$, we get a map $B \to A((\xi_i))$ which
induces injections
\[
\mg{g} \otimes_k B \to  \mg{g} \otimes_k \Bigl
(\bigoplus_{i=1}^nk((\xi_i)) \Bigr ) \otimes_k A \to \hat{\mg{g}}_n
\otimes_k A \ .
\]
The fact that the sum of the residues of a $1$-form on a smooth
projective curve is zero implies that this injection makes $\mg{g}
\otimes_k B $ into a sub Lie algebra of $\hat{\mg{g}}_n \otimes_k A$.

By linearity $\mc{H}_{\ovl} \otimes_k A$ is a representation of
$\hat{\mg{g}}_n \otimes_k A$ and the sheaf of conformal blocks
$V_{\mc{C}}(\ov{p},\ovl)$ is defined to be the quasi-coherent sheafon
$S$, which is in fact locally free of finite rank, corresponding to
the $A$-module
\[
\mc{H}_{\ovl} \otimes_k A / (\mg{g} \otimes_k B)\cdot (\mc{H}_{\ovl} \otimes_k A ) \ .
\]
One may give an intrinsic description of this module without using the
isomorphisms $\eta_i$; we only use this implicitly so we refer the
reader to \cite{looijenga-conformal} for the details. Since such
isomorphisms always exist Zariski locally on any $S$, we may sheafify
the above construction to define a locally free sheaf
$V_{\mc{C}}(\ov{p},\ovl)$ for general $S$.

Note that if $\mc{C} = \mc{C}' \cup \mc{C}''$ is a disjoint union of
two families of semi-stable curves over $S$ then we have
$V_{\mc{C}}(\ov{p},\ovl) \cong V_{\mc{C}'}(\ov{p}',\ovl') \otimes
V_{\mc{C}''}(\ov{p}'',\ovl'')$, where $\ov{p}'$ denotes the tuple of
sections lying in $\mc{C}'$, $\ovl'$ denotes the corresponding tuple
of weights and similarly for $\ov{p}''$ and $\ovl''$.

\subsection{Basic properties}
\label{sec:basic} \marginpar{sec:basic}

Keeping the previous notation, let $\ov{q}= (q_1,\dots,q_m)$ with the
$q_j$ also sections of $\pi$ with images which are disjoint, disjoint
from all the $p_i(S)$ and also contained in the smooth locus of
$\pi$. Let $\ov{pq} = (p_1,\dots,p_n,q_1,\dots,q_m)$, $\ov{\lambda
  0_m} = (\l_1,\dots,\l_n, 0,\dots,0)$, where we have $m$ zeros.  We
then have the following \cite[2.3.2]{sorger-verlinde},
\cite[Proposition 3.4]{looijenga-conformal}:
\begin{prop}[Propagation of vacua]
  \label{prop:prop} \marginpar{prop:prop} There is a natural
  isomorphism $V_{\mc{C}}(\ov{p}, \ovl) \sr{\sim}{\lr}
  V_{\mc{C}}(\ov{pq},\ov{\lambda 0_m})$. Moreover, these isomorphisms
  are compatible as $m$ varies.
\end{prop}

We now drop the condition that $\mc{C} \bs \cup_i p_i(S)$ is affine
over $S$. For $S' \to S$ any morphism, we let $\mc{C}' =
\mc{C}\times_S S'$, $\pi':\mc{C}' \to S'$ the induced morphism and
$p_i'$ the induced sections. We may find an etale cover $S' \to S$ and
$m$ sections $\ov{q'} = (q_1',\dots,q_m')$ of $\pi'$ as above so that
$\mc{C}' \bs( \cup_i p_i'(S') \cup_j q_j'(S'))$ is affine. Therefore,
$ V_{\mc{C}'}(\ov{p'q'},\ov{\lambda 0_m} )$ is defined and is a
locally free sheaf on $S'$. By Proposition \ref{prop:prop} we get
natural descent data for this sheaf with respect to the morphism $S'
\to S$ and we define $V_{\mc{C}}(\ov{p},\ovl)$ to be the descent of
this sheaf to $S$.

To see that this is independent of the choice of the etale cover $S'
\to S$ and sections $q_j'$ we use the same proposition. Note that it
may happen that if we use two sets of auxiliary sections then all the
sections in the union may not be pairwise disjoint. However, since the
sections are disjoint or equal at all generic points of the base $S$,
we get get canonical isomorphisms over a dense open set of $S$. Using
further auxiliary sections disjoint from all the original ones, etale
locally around each point of $S$, one sees that these isomorphisms
extend over all of $S$.

\bigskip

Since the sheaves of conformal blocks are compatible with base change,
the above discussion shows that for any $n$-tuple $\ovl \in
{P_{\ell}}^n$ one has a well defined sheaf of conformal blocks on
$\mgnb$, the moduli stack of stable $n$-pointed curves of
genus $g$.

Let $\ov{p}$, $\ovl$ and $\ov{q}$ be as above with $m=2$.  Let
$\pi_{\mc{D}}:\mc{D} \to S$ be the family of curves obtained by gluing
$\mc{C}$ along the sections $q_1$ and $q_2$. The sections $p_i$ induce
sections of $\pi_{\mc{D}}:\mc{D} \to S$ which we also denote by
$p_i$. For $\mu \in P_{\ell}$, let $\ov{\lambda \mu} =
(\l_1,\dots,\l_n,\mu, \mu^*)$.  We then have \cite[Proposition
4.1]{looijenga-conformal}, \cite[2.4.2]{sorger-verlinde}
\begin{prop}[Factorisation formula]
\label{prop:fact} 
There are natural isomorphisms
\[
V_{\mc{D}}(\ov{p}, \ovl) \sr{\sim}{\lr} \bigoplus_{\mu \in P_{\ell}}
V_{\mc{C}}(\ov{pq}, \ov{\l\mu}) \ .
\]
\end{prop}

\smallskip

Let $\pi:\mc{C} \to S$ be a family of semi-stable curves and $\ov{p}
=(p_1,\dots,p_{n+1})$ an $n+1$-tuple of sections making
$(\mc{C},\ov{p})$ into a family of stable $n+1$-pointed curves for
some $n \geq 0$. Let $\ovl = (\l_1,\dots, \l_n,0) \in
{P_{\l}}^{n+1}$. If we let $\ov{p}' = (p_1,\dots,p_{n})$, then the
family of $n$ pointed curves $(\mc{C},\ov{p}')$ may no longer be
stable. We have a stablisation morphism $\sigma:\mc{C} \to \mc{C}'$
over $S$ which contracts certain smooth rational curves in each fibre
of $\pi$ so that $(\mc{C}', \ov{p}')$ is a stable family of
$n$-pointed curves.  Set $\ov{\l}' = (\l_1,\dots,\l_n)$.
\begin{lem} \label{lem:contract} 
With notation as above, there is a natural isomorphism
of sheaves
\[
V_{\mc{C}'}(\ov{p}',\ov{\l}') \sr{\sim}{\lr} V_{\mc{C}}(\ov{p},\ovl) \ .
\]

\end{lem}

\begin{proof}
  The stabilisation morphism $\sigma$ induces a morphism $\mc{C} \bs
  \cup_{i=1}^{n+1}p_i(S) \to \mc{C}' \bs \cup_{i=1}^n p_i(S)$ which
  gives rise to the map of sheaves.  Since the bundles of conformal
  blocks are locally free, we may check that this map is an
  isomorphism fibrewise.

  There are two types of rational curve components that may get
  contracted by the stabilisation morphism. One type is a smooth
  rational curve with two marked points, one of them being the
  $n+1$'st 
  and intersecting the union of the
  other components in a single point. When this is contracted the
  point of intersection becomes the other marked point.  The other
  type is a smooth rational curve meeting the union of the other
  components in two points and containing the $n+1$'st marked point,
  this being the unique marked point on it.

  In both these cases, a direct application of the factorisation
  formula shows that the map is indeed an isomorphism.
\end{proof}

As a consequence of the above we see the the bundles of conformal
blocks are compatible with the natural morphisms among the moduli
stacks of stable curves. More precisely, we have

\begin{prop} \label{prop:compat} 
For $\ovl \in {P_{\ell}}^n$, let $\mb{V}_{g,n,\ovl}$ denote the sheaf
of conformal blocks associated to $\ovl$ on $\mgnb$.
\begin{enumerate}
\item Let $f_i: \mgn \to \omc_{g,n-1}$ be the
  morphism given by forgetting the $i$'th marked point.  If $\l_i = 0$
  for some $i$, then $\mb{V}_{g,n,\ovl} \cong f_i^*
  (\mb{V}_{g,n,\ov{\l}'})$, where $\ov{\l}'$ is obtained from $\ovl$
  by deleting $\l_i$.
\item Let $ \gamma: \omc_{g_1,n_1+1} \times
  \omc_{g_2,n_2+1} \to \omc_{g_1 + g_2,n_1+n_2}$ be the
  gluing morphism, where we glue along the last marked point for each
  factor. Then for $\ovl \in {P_{\ell}}^{n_1+n_2}$ we have a natural
  isomorphism
  \[
  \gamma^*(\mb{V}_{g_1 + g_2,n_1+n_2,\ovl}) \sr{\sim}{\to}
  \bigoplus_{\mu \in P_{\ell}} \mb{V}_{g_1,n_1+1,\ovl^1\mu} \otimes
  \mb{V}_{g_2,n_2 + 1,\ovl^2\mu^*}
  \]
  where $\ovl^1\mu := (\l_1,\dots,\l_{n_1},\mu)$ and $\ovl^2  \mu^*
  := (\l_{n_1+1},\dots,\l_{n_1+n_2},\mu^*)$.
\item Let $\gamma:\omc_{g-1,n+2} \to \mgnb$ be the
  gluing morphism where we glue the last two marked points to each
  other. Then for $\ovl \in {P_{\ell}}^n$ we have an isomorphism
  \[
  \gamma^*(\mb{V}_{g,n,\ovl}) \sr{\sim}{\to} \bigoplus_{\mu \in
    P_{\ell}} \mb{V}_{g-1,n+2, \ov{\l\mu}}
  \]
  where $\ov{\l\mu} := (\l_1,\dots,\l_n,\mu,\mu^*)$.
\end{enumerate} 
\end{prop}

\begin{proof}
  (1) follows from Lemma \ref{lem:contract} and (2), (3) follow from
  Proposition \ref{prop:fact}
\end{proof}

In this paper we shall mostly be concerned with the case $g=0$. In
this case, $\mnb$ is a smooth projective variety and we shall from now
on denote the locally free sheaf $\mb{V}_{0,n,\ovl}$ on $\mnb$ simply
by $\vl$, its determinant line bundle by $\dl$ and its rank by
$r_{\ovl}$.

\begin{lem}
\label{lem:gen} 
All the $\vl$ are generated by their global sections, therefore so are
all the $\dl$. In particular, the $\dl$  are nef line bundles.
\end{lem}

\begin{proof}
  With notation as in the beginning of this section, by the
  construction of the sheaves of conformal blocks there is always a
  natural map from the constant vector bundle on $S$ with fibre
  $(\otimes_i V_{\l_i})_{\mg{g}}$ to $V_{\mc{C}}(\ov{p}, \ovl)$ which
  is induced by the inclusion of $V_{\l}$ in $\mc{H}_{\l}$. It follows
  from \cite[Proposition 3.5.1]{ueno-conformal} that this map is
  surjective if the fibres of the family are smooth curves of genus
  $0$. The proof only uses the fact that given any point $x \in
  \P^1(k)$ there exists $f \in \Gamma(\P^1 - \{x\}, \mc{O}_{\P^1 -
    \{x\}})$ which has a simple pole at $x$. If $C$ is a semi-stable
  curve of genus $0$ and $x \in C(k)$ is a smooth point, then there
  exists $f \in \Gamma(C - \{x\},\mc{O}_{C - \{x\}})$ having a simple
  pole at $x$, so the proof of \cite[Proposition
  3.5.1]{ueno-conformal} extends to the case where the fibres are
  semi-stable curves of genus $0$.
\end{proof}

\begin{rem} 
  \label{rem:allg}
  The map from $(\otimes_i V_{\l_i})_{\mg{g}}$ to $V_{\mc{C}}(\ov{p},
  \ovl)$ is in general not surjective if the fibres have genus $g>0$.
  In fact, in this case it follows from Theorem \ref{thm:m11} that the
  determinants of conformal blocks on $\mnb$ are often
  not even nef.
\end{rem}

Recall that $\mnb \bs \mn$ is a divisor with simple normal
crossings, so $\mnb$ has a natural stratification by smooth
strata. Keel \cite{keel-m0nbar} has shown that the Chow groups of
$\mnb$ are generated by the classes of the closures of the
(irreducible componenents) of the strata and the Chow groups are equal
to the Chow groups modulo numerical equivalence.  Thus, one way of
computing the class of $\dl$ in $\pic(\mnb)$ is by computing
the degree of $\vl$ restricted to the one dimensional strata,
called vital curves in \cite{keel-mckernan}.

Classes of vital curves modulo numerical equivalence correspond to
partitions $\{1,\dots,n\} =  \sqcup_{k=1}^4 N_k$, with
$|N_k| = n_k > 0$.  Given such a partition, let $F$ be the family of
$n$-pointed genus $0$ curves given by gluing, for each $k=1,2,3,4$, a
fixed $n_k + 1$-pointed curve $C_k$ of genus $0$ at the last marked
point along the $k$-th section of the universal family over
$\mfb$. If $n_k = 1$ for some $k$, we do not glue any curve at
the $k$'th section. This gives rise to a family of stable $n$-pointed
curves of genus $0$ such that the class in the Chow group of the image
of $\mfb$ in $\mnb$ by the classifying map for this
family is independent of the choice of the glued curves.

Given a partition as above and $\ov{\mu} = (\mu_1,\mu_2,\mu_3,\mu_4)
\in {P_{\ell}}^4$, let $\ov{\l}\mu_k^*$ be the $n_k + 1$-tuple
$(\l_{i_1},\dots,\l_{i_{n_k}},\mu_k^*)$ where $N_k = \{i_1,i_2,\dots,
i_{n_k}\}$.  Since in the construction of $F$ the attached curves do
not vary in moduli, it follows from the factorisation formula applied
four times that we have 
\begin{prop}
\label{prop:deg} 
\[
\deg(\mb{V}_{\ovl}|_F) = \sum_{\ov{\mu} \in {P_{\ell}}^4}
\deg(\mb{V}_{\ov{\mu}})\ \prod_{k=1}^4 r_{\ov{\l}\mu_k^*} \ .
\]
\end{prop}

The ranks of the bundles of conformal blocks can be computed from the
Verlinde formula or inductively from the case $n=3$, so it follows
that to compute the determinant of $\vl$, equivalently its degree on
any vital curve, it suffices to consider the  $n=4$ case.

\begin{rem} 
\label{rem:deg}
For arbitrary $g$, the closures of the $1$-dimensional strata of
$\overline{\mr{M}}_{g,n}$ are called F-curves and there is a simple
way of enumerating all these \cite[Theorem 2.2]{GKM}. It is known that
their classes generate $N_1(\mgnb)_{\Q}$, the space of $1$-cycles
modulo numerical equivalence, so it follows from the description of
F-curves and the factorisation formula that to compute
$c_1(\mb{V}_{g,n,\ovl})$ for all $g,n$ it suffices to be able to
compute it in the cases $(g,n) = (0,4)$ or $(1,1)$.
\end{rem}

\subsection{The KZ/Hitchin/WZW connection} 
\label{subsec:kzconn} 

The sheaves of conformal blocks associated to a smooth family of
curves have a natural flat projective connection called the WZW or
Hitchin connection; this is one of the main ingredients in our
computation of the Chern classes. However, for that purpose it is
important to lift this to a flat connection. For arbitrary $g$ this is
not always possible globally but it is for $g=0$ or $1$. The lift is
not canonical and depends on some auxiliary choices and our first goal
in this section is to understand this dependence explicitly.

For the sake of simplicity, we assume that $S$ is affine and $A =
\Gamma(S,\mc{O}_S)$ as before. Then the first choice is that of formal
parameters, \emph{i.e.}, the isomorphisms $\eta_i$, along each section
of the family of curves $\pi:\mc{C} \to S$ with an $n$-tuple of
sections $\ov{p}$. The second is that of a suitable symmetric
bidifferential $\omega$ on the family, see
\cite[p.~16]{ueno-conformal}. Given this data, the action of a vector
field $D$ on $S$ on is defined as follows \cite[Section
5.1]{ueno-conformal}: We first lift $D$ to a vector field, also
denoted by $D$, on $\mc{C} \bs \cup_i p_i(S)$. For each $i$, we write
$(\eta_i^{-1})^*(D) = D_{i,hor} + D_{i,vert}$ where $D_{i,hor}(\xi) =
0$ and $D_{i,vert}$
kills $A$.  Then we have the Sugawara operators $T(D_{i,vert})$ which
act on $\mc{H}_{\l_i}$ and hence on $\mc{H}_{\ovl}$ by acting on the
$i$'th factor. Morevoer, the bidifferential $\omega$ gives elements
$a_{\omega,i}(D_{i,vert}) \in A$ \cite[(5.4)]{ueno-conformal} whose
sum over all $i$ we write simply as $a_{\omega}(D_{vert})$; we do not
recall the full definition here but note that it is of the form
$\tfrac{\ell \dim \mg{g}}{\ell + h^{\vee}}f$, with $f \in A$ 
not depending on any data associated to $\mg{g}$.

The action of $D$ on $V_{\mc{C}}(\ov{p},\ovl)$, which we denote by
$\nabla_D$, is induced by the action on $\mc{H}_{\ovl} \otimes_k A$
given by
\[
\nabla_D(v \otimes f) = D_{hor}(v \otimes f) + \Bigl (\sum_{i=1}^n
T(D_{i,vert})(v) \Bigr )\otimes f - v \otimes a_{\omega}(D_{vert}) \cdot f \
,
\]
where $v \in \mc{H}_{\ovl}$ and $f \in A$ and the action of $D_{hor}$
is given by coordinatewise differentiation, \emph{i.e.}, it acts on the
$A$ component of $\mc{H}_{\ovl} \otimes_k A$. We note that the
energy-momentum tensor $T$, whose definition we recall below, as well
as the operator $a_{\omega}$ depend on the choice of
parameters. Moreover, the first term also depends on this choice since
this is implicit in the tensor product decomposition.

We now choose some different isomorphisms $\eta_i'$ and compute how
the connection changes; in the following we denote all terms defined
using these new isomorphisms with a $'$. By definition, the difference
of the two actions $\nabla_D - \nabla_{D}'$ is given by the operator
\[
D_{hor} - D_{hor}' + T(D_{vert}) - T'(D_{vert}')
-(a_{\omega}(D_{vert}) - a_{\omega}'(D_{vert}')) \ ,
\]
where to simplify the notation we have suppressed the sum over $i$.
Adding and subtracting suitable terms, this is equal to
\[
D_{hor} -D_{hor}' + T(D_{vert} - D_{vert}') + T(D_{vert}') -
T'(D_{vert}') -(a_{\omega}(D_{vert} - D_{vert}') +
a_{\omega}(D_{vert}') - a_{\omega}'(D_{vert}')) \ .
\]
It follows from the discussion before Lemma 2.13 of
\cite{looijenga-conformal}, more precisely by using the version for
integrable representations \cite[Lemma 2.13]{ueno-conformal} rather
than the Fock representation,
that
\[
T(D_{vert} - D_{vert}') = -(D_{hor} - D_{hor}') + U
\]
where $U$ is the operator acting on $\mc{H}_{\ovl} \otimes_k A$ by the
scalar (see below), giving the action of $T(D_{vert} - D_{vert}')$ on
$\otimes_i V_{\l_i} \subset \mc{H}_{\ovl}$.

Using the definition in \cite[(5.4)]{ueno-conformal} one sees that
$a_{\omega}(D_{vert} - D_{vert}') = 0$. The change of variables
formula for the energy-momentum tensor $T$
(\cite[8.2.2]{frenkel-benzvi} or Theorem 3.4.3 (2) of
\cite{ueno-conformal}) and for the projective connection associated to
the bidifferential $\omega$ (\cite[Theorem 1.115]{ueno-conformal})
implies that both these terms change by the Schwarzian derivative when
they are applied to the same vector field but different coordinates
are used. Thus the changes in these terms cancel out.

It follows that the difference in the connections is given simply by
the operator $U$.  As a function of $\ovl$ it depends only on the
$c(\l_i)$. To see this, we first recall the definition of the
energy-momentum tensor $T$.

We use the notation
\begin{align*}
X(n) & = X \otimes \xi^n, \ X \in \mg{g} \\
X(z) & = \sum_{n\in \Z} X(n)z^{-n-1}
\end{align*}
The normal ordering $: :$ is defined by
\begin{equation*}
:X(n) Y(m): \ \  = 
\begin{cases} 
X(n)Y(n) & n<m,\\
\tfrac{1}{2}(X(n)Y(m) + Y(m)X(n)) & n=m, \\
Y(m)X(n) & n>m \ .
\end{cases} 
\end{equation*}
Put
\[
L_n = \frac{1}{2(\ell + h^{\vee})}\sum_{m \in \Z}\sum_{\kappa}:X_{\kappa}(m)X_{\kappa}(n-m):
\]
and
\[
T(z) = \sum_{n \in \Z}L_nz^{-n-2}\ .
\]
The $L_n$'s are called Virasoro operators and act on $\mc{H}_{\l}$.


For $D = f(z)\frac{d}{dz}$ with $f \in k((z))$, put
\[
T(D) = \res_{z=0}(T(z)f(z)dz) \ .
\]
Then $T(D)$ acts on $\mc{H}_{\l}$. 

If $f = \sum_{j=1}^{\infty} a_jz^j$ then $T(D) = \sum_{j=1}^{\infty}
a_jL_{j-1}$. The formula for the $L_n$ shows that this preserves
$V_{\l} \subset \mc{H}_{\l}$ and acts on it as the operator
$\tfrac{a_1\Omega}{2(\ell + h^{\vee})}$, hence by multiplication by
$\tfrac{a_1c(\l)}{2(\ell + h^{\vee})}$.

\smallskip

We summarize the above discussion in 
\begin{lem} \label{lem:change} 
\begin{enumerate}
\item For each $i$, $D_{i,vert} - D_{i,vert}'$ is of the form
  $f_i(\xi)\tfrac{d}{d\xi}$ with $f(\xi) = \sum_{j=1}^{\infty}a_j^i\xi^j$
  and $a_j^i \in A$.
\item $\nabla_D - \nabla_D'$ is given by multiplication by 
 \[(2(\ell + h^{\vee}))^{-1}\sum_i
  a_1^ic_1(\l_i) \ .
  \]
\end{enumerate}
\end{lem}
\begin{proof}
The first part follows easily from the definitions and the second from
the discussion preceding the statement of the lemma.
\end{proof}

The connection on the sheaf of conformal blocks has logarithmic
singularities along the boundary divisors for a degenerating family of
smooth curves. Our formula for the Chern classes will be obtained from
considerations of the residues of the connection along the boundary. These
are known to have a simple description which we now recall:

Let $C'$ be a smooth projective curve with distinct points $q_1,q_2$
in $C'(k)$ and let $C$ be the nodal curve obtained by gluing these
points together. We assume that $C$ is connected. Given local
parameters $\xi_{q_1}$ and $\xi_{q_2}$ at $q_1$ and $q_2$ respectively,
there exists a natural smoothening $\pi:\mc{C} \to S$ of $C$. Here $S
= \spec(k[[t]])$, $\pi$ is proper and flat, the special fibre is $C$
and the generic fibre is smooth. Moreover, if $q$ denotes the node on
$C$, the formal completion of $C - \{q\}$ in $\mc{C}$ is naturally
isomorphic to the formal completion of $C - \{q\}$ in $C
-\{q\}\times_k \spec(k[[t]])$. In particular, any smooth point of $C(k)$
extends naturally to a section of $\pi$ and a local parameter for such
a point extends to a local equation for the corresponding section.
This smoothening can be constructed easily using deformation theory
and the Grothendieck existence theorem; we shall give a more explicit
description in the cases we actually use.

Now suppose $p_1,p_2,\dots,p_n$ are smooth rational points on the
special fibre $C$ which we extend to sections as described above and
let $\ovl \in {P_{\ell}}^n$.  Let $D$ denote the vector field $t
\tfrac{d}{dt}$ on $S$. We lift $D$ to a rational vector field on
$\mc{C}$ which is regular outside the sections. It follows from the
calculations in \cite[Section 4]{looijenga-conformal} that the
Sugawara action of $D$ induces an operator on
$V_{\mc{C}}(\ov{p},\ovl)$ which restricts to the operator on
$V_C(\ov{p},\ovl)$ given as follows:
\begin{prop} 
\label{prop:res}
Under the natural identification of $V_C(\ov{p},\ovl)$ with
$\oplus_{\mu \in P_{\ell}}V_{C'}(\ov{pq},\ov{\lambda \mu})$ given by
the factorisation formula, the operator acts on each summand
$V_{C'}(\ov{pq},\ov{\lambda \mu})$ by multiplication by
$\tfrac{c(\mu)}{2(\ell + h^{\vee})} + b$ with $b$ being given by the
Sugawara action of $D_0$, the restriction of $D$ to $C'$.
\end{prop}
Note that since $D_0$ is a vertical vector field, $b$ does not depend
on $\mu$ and can be computed using a bidifferential, see \cite[Lemma
5.1]{ueno-conformal}. In \cite[Theorem 4.5]{looijenga-conformal} it
appears to be claimed that $b$ is always zero but we do not see
why. However, we will see that it is zero in the cases we consider.

Note that in the above we have not included the $a_{\omega}$ term in
the action. This will vanish in the case $g=0$ that we consider but
not for $g>0$.

\section{The case $g=0$}
\label{sec:g0} 

In this section we derive our main formulae for the Chern classes in
genus $0$ by explicitly computing the residues of the KZ connection
on $\mb{V}_{\ovl}$ along all the boundary divisors in $\mnb$.

\subsection{} 
\label{subsec:kzres} 

In order to carry out our computations we shall need an explicit
description of $\mn$ and explicit equations for the $n$ sections.
Therefore, we identify $\mn$ with the open subset of $\A^{n-3}$
given by
\[
\{(z_1,z_2,\dots,z_{n-3}) \in \A^{n-3} \ |\text{ $\ z_i \neq 0,1$ for all
  $i$ and $z_i \neq z_j$ for $i \neq j$}\} \ .
\]
The universal family of marked curves is given by $\mn\times
\P^1$ with the $n$ ordered sections given by the $n$ morphisms
$\mn \to \P^1$, $(z_1,z_2,\dots,z_{n-3}) \mapsto
z_1,z_2,\dots,z_{n-3},0,1,\infty$.  Letting $x$ be the coordinate on
$\P^1$, the sections are given by the equations $x
=z_1,z_2,\dots,z_{n-3},0,1$ and $1/x = 0$.

For the rest of this section, we fix a simple Lie algebra $\mg{g}$, a
level $\ell$ and $\ovl = (\l_1,\l_2,\dots,\l_n) \in {P_{\ell}}^n$. If
we use the bidifferential $\omega(x,y) = \frac{dxdy}{(x-y)^2}$ on
$\P^1 \times \P^1$, by \cite[Section 5.1]{ueno-conformal} we get a
well defined flat connection on $\mb{V}_{\ovl}$ restricted to $\mn$
with regular singularities which we call the KZ connection.  In fact,
since all equations for the sections are fractional linear and the
Schwarzian derivative of such a function is $0$, it follows from the
definition \cite[(5.40)]{ueno-conformal} that
$a_{\omega,i}(D_{i,vert}) = 0$ for all $i$ for any vector field $D$ on
$\mn$.
To make things precise, given a vector field $D$ on $\mn$ we will
always lift it to a vector field on $\mn \times \P^1$ so that it is
constant on the fibres.

We now compute the residues of the KZ connection along the boundary
divisors in $\mnb$.  We will first choose good local coordinates in
order to be able to apply Proposition \ref{prop:res} and then compute
the change in the connection, hence the change in the residue, by
applying Lemma \ref{lem:change}. All the local coordinates that we use
will be fractional linear so it follows from \cite[Lemma
5.1]{ueno-conformal} that the constant $b$ occurring in Proposition
\ref{prop:res} is always zero. What we need then is to compute the
functions $a_1^i$ occurring in Lemma \ref{lem:change}.

Recall that the boundary divisors are parametrised by partitions
$\{1,2,\dots,n\} = A \sqcup B$ with $|A|, |B| \geq 2$. In the
coordinates above, they correspond to the exceptional divisors in the
blowup of the following loci:
\begin{enumerate}
\item For $\emptyset \neq S \subset \{1,2,\dots,n-3\}$, the locus in
  $\A^{n-3} \supset \mn$ given by the equations $\{z_i = 0\}_{i
    \in S}$.
\item For $\emptyset \neq S \subset \{1,2,\dots,n-3\}$, the locus in
  $\A^{n-3} \supset \mn$ given by the equations $\{z_i = 1\}_{i
    \in S}$.
\item For $\emptyset \neq S \subset \{1,2,\dots,n-3\}$, the locus in
  $(\P^1)^{n-3} \supset \A^{n-3} \supset \mn$ given by the
  equations $\{1/z_i = 0\}_{i \in S}$.
\item For $S \subset \{1,2,\dots,n-3\}$, with $|S| \geq 2$, the locus
  in $\A^{n-3} \supset \mn$ given by the equations $\{z_i =
  z_j\}_{i,j \in S}$.
\end{enumerate}

Globally, each of these divisors is the image of the embedding of
$\ov{\mr{M}}_{0,r+2} \times \ov{\mr{M}}_{0,n-r}$ into $\mnb$ by a
suitable gluing map, where $r = |S|$ in the first three cases and $r =
|S| -1$ for the last case. By Proposition \ref{prop:fact}, the
restriction of $\vl$ to each of these divisors is a sum $\oplus_{\mu
  \in P_{\ell}} \V_{\ovl'{\mu}} \otimes \V_{\ovl''{\mu^*}}$, where
${\ovl'{\mu}}$ (resp.~${\ovl''{\mu^*}}$) is obtained by restricting
$\ovl$ and attaching $\mu$ (resp.~$\mu^*$) at the glued point of the
first (resp.~second) component. The residue of the KZ connection along
these divisors---this is an endomorphism of the restricted
bundle---preserves this direct sum decomposition and moreover acts on
each summand by a scalar which we shall now determine.

In what follows, we shall denote $\del / \del z_i$ by $\del_i$ and
$\del / \del x$ by $\del_x$.

\subsubsection{} 
\label{sec:kzres1}

Boundary divisors of type (1): For ease of notation, we
shall assume $S = \{1,2\dots,r\}$ since the general case follows from
this by permuting coordinates.  An open set in the blowup is given by
$U \cong \A^{n-3}$ with coordinates
$t,w_2,\dots,w_r,z_{r+1},\dots,z_{n-3}$ with the map to $\mn$ given by
$(t,w_2,\dots,w_r,z_{r+1},\dots,z_{n-3}) \mapsto
(t,tw_2,\dots,tw_r,z_{r+1},\dots,z_{n-3})$ and the exceptional divisor
$B$ is given by $t=0$. The universal family in a neighbourhood of the
generic point of the exceptional divisor is given by blowing up the
locus given by $t = x = 0$ in $U \times \P^1$, so over the generic
point there are two components isomorphic to $\P^1$ meeting
transversally in a single point.

Let $y := t/x$. The $n$ sections defined over $\mn$ extend to
sections of this family as follows, where $w_1:=1$:
\begin{itemize}
\item The sections given by $x=z_{r+1}\dots,z_{n}$, $x=1$ and $1/x =
  0$ are given by the same equations.
\item The section given by $x = 0$ is given by $1/y= 0$.
\item The sections given by $x = z_i$, $1\leq i \leq r$, are given by
  $y =w_i^{-1}$.
\end{itemize}

Replacing $k$ by $k(w_2,\dots,w_r,z_{r+1},\dots,z_{n-1})$, it follows
from Proposition \ref{prop:res} that the residue of the
connection obtained using the above coordinates along this divisor and
the new equations for the sections is given by the endomorphism of
$\vl|_{B}$ which acts by multiplication by $c(\mu)/2(\ell + h^{\vee})$
on the summand $\V_{\ovl'{\mu}} \otimes \V_{\ovl''{\mu^*}}$.

Let $D= \del / \del t$ and lift it to a derivation on the universal
family over $U$ with trivial action in the fibre direction.  To
compare the connection in the new coordinates as above with the KZ
connection we must compute the functions $a_1^i$, $i=1,2,\dots,n$
occurring in Lemma \ref{lem:change}.  As before, we write $D =
D_{vert} + D_{hor} = D_{vert}' + D_{hor}'$ where the $'$ denotes the
decomposition with respect to the new coordinates.

\begin{itemize}
\item For the sections given $x=z_{r+1}\dots,z_{n}$, $x=1$ and $1/x =
  0$, $D_{vert} = D_{vert}' = 0$, hence $a_1 = 0$.
\item For the section given by $x =0$ (this is the $n-2$'nd according
  to our numbering), we have $D_{vert} = 0$ but $D_{vert}' =
  -(x/t)\del_x$ so $a_1 = 1/t$.
\item For the sections with equations $x = z_i$, $1 \leq i \leq r$, by
  substituting $z_i = tw_i$ we see that $D_{vert} = -w_i\partial_x$
  whereas $(\del_t + (x/t)\del_x)((t/x) - w_i^{-1}) = 0$ so $D_{vert}'
  = -(x/t)\del_x = -(\frac{x-z_i}{t} + w_i)\partial_x$. It follows
  that $a_1 = 1/t$.
\end{itemize}

Adding up all the terms, we see that the residue of the KZ connection
along this divisor is the endomorphism of $\vl|_B$ which acts by
multiplication by 
\begin{equation}
\label{eq:kzres1}
\frac{c(\mu) - c(\l_{n-2}) - \sum_{i \in S} c(\l_i)}{2(\ell + h^{\vee})}
\end{equation}
on the summand $\V_{\ovl'{\mu}} \otimes \V_{\ovl''{\mu^*}}$ for each
$\mu \in P_{\ell}$.

\subsubsection{} 
\label{sec:kzres2}

Boundary divisors of type (2): The change of coordinates
given by $z_i \mapsto 1 - z_i$, $i =1,2,\dots,n-3$ and $x \mapsto 1
-x$ on $\mn\times \P^1$ preserves the equations of all the sections
except for $x=0$ and $x=1$ which it interchanges and $1/x = 0$ which
becomes $1/(1-x) = 0$. Moreover, it sends the locus given by $\{z_i =
1\}_{i \in S}$ to the locus given by $ \{z_i=0\}_{i \in S}$. Since
$\del_i (1/x) = \del_i (1/(1-x)) = 0$ for all $i$, the KZ connection
does not change if we replace the equation $1/x = 0$ by the equation
$1/(1-x) = 0$. Using \eqref{eq:kzres1} we then see that the residue of
the KZ connection along a divisor $B$ of type (2) is the endomorphism
of $\vl|_B$ which acts by multiplication by 
\begin{equation}
\label{eq:kzres2}
\frac{c(\mu) - c(\l_{n-1}) - \sum_{i \in S} c(\l_i)}{2(\ell + h^{\vee})}
\end{equation}
on the summand $\V_{\ovl'{\mu}} \otimes \V_{\ovl''{\mu^*}}$ for each
$\mu \in P_{\ell}$.

\subsubsection{} 
\label{sec:kzres3}

Boundary divisors of type (3): The change of coordinates given by $z_i
\mapsto 1/z_i$, $i=1,2,\dots,n-3$ and $x \mapsto 1/x$, preserves the
sections given by $x = z_i$, $i=1,2,\dots,n-3$ and $x = 1$ and
switches the sections given by $x = 0$ and $1/x = 0$. Moreover, it
maps the locus given by the equations $\{1/z_i = 0\}_{i\in S}$ to the
locus given by the equations $\{z_i=0\}_{i \in S}$ so we would like to
use the computation for type (1) boundary divisors to compute the
residue.  However, since the equations of the sections are not
preserved we must again compute the change in the connection caused by
the change of coordinates.

For $i \in \{1,2,\dots,n-3\}$ set $D = \del_i$. The old as well as the
new equations for all the sections except for the $i$'th one are
killed by $\del_i$, so $D_{vert} = D_{vert}' = 0$ along them.  For the
$i$'th section we have $D_{vert} = -\del_x$ and since $(\del_i -
(x^2/z_i^2)\del_x)(1/x - 1/z_i) = 0$ we have $D_{vert}' =
-x^2/z_i^2\del_x$. Since
\[
-\frac{x^2}{z_i^2} = - \frac{(x-z_i)^2 +2z_i(x-z_i)}{z_i^2} - 1
\]
it follows that $a_1 = 2/z_i$. 

To compute the difference in the residue, we must compute the
difference in the action of the vector field $\del / \del t$ in the
notation of \S \ref{sec:kzres1}.  Since $\del / \del t = \sum_{i \in
  S}w_i \del _i$ it follows from Lemma \ref{lem:change} that this is
given $\frac{\sum_{i \in S} c(\l_i)}{t(\ell + h^{\vee})}$.

It then follows from \eqref{eq:kzres1} that the residue of the KZ
connection along this divisor is the endomorphism of $\vl|_B$ which
acts by multiplication by 
\begin{equation}
\label{eq:kzres3}
\frac{c(\mu)  - c(\l_n) - \sum_{i \in S}c(\l_i)}{2(\ell + h^{\vee})}
+ \frac{\sum_{i \in S} c(\l_i)}{\ell + h^{\vee}}
= \frac{c(\mu) - c(\l_n) + \sum_{i \in S} c(\l_i)}{2(\ell + h^{\vee})}
\end{equation}
on the summand $\V_{\ovl'{\mu}} \otimes \V_{\ovl''{\mu^*}}$ for each
$\mu \in P_{\ell}$.

\subsubsection{} 
\label{sec:kzres4}

Boundary divisor of type (4): Again, for ease of notation we shall
suppose $S = \{1,2,\dots,r\}$ for some $r$, $2 \leq r \leq n-3$. Then
an open subset of the blowup may be identified with $\A^{n-3}$ with
coordinates $s,t,w_3,\dots,w_r,z_{r+1},\dots,z_{n-3}$ so that the map
to $\mn$ is given by $(s,t,w_3,\dots,w_r,z_{r+1},\dots,z_{n-3})
\mapsto (s, s + t, s+ tw_3, \dots, s+ tw_r, z_{r+1}, \dots,z_{n-3})$
and the exceptional divisor $B$ is given by $t=0$.  The universal
family in a neighbourhood of $B$ is then given by blowing up the locus
given by $t = x-s = 0$, where $x$ is the coordinate on $\P^1$ as
before.

Let $y = t/(x-s)$. All $n$ sections extend to sections of the universal
family in a neighbourhood of the generic point of $T$ with equations
given as follows, where $w_2 := 1$:
\begin{itemize}
\item $x - z_1 = 0$ is replaced by $1/y=0$.
\item $x - z_i=0$ is replaced by $y - w_i^{-1}=0$ for $2\leq i \leq r$.
\item The equations for the remaining sections are unchanged.
\end{itemize}

As in \S \ref{sec:kzres1} we now decompose $\del_t$ as $D_{vert} +
D_{hor}$ and $D_{vert}' + D_{hor}'$ and compare the two:
\begin{itemize}
\item For  $x - z_1 = 0$, $D_{vert} = 0$ and $D_{vert}' = -(x-z_1)/t\del_x$ so
  $a_1 = 1/t$.
\item For $x - z_i=0$, $2\leq i \leq r$, using $z_i = s+tw_i$ we get $D_{vert} = -w_i\del_x$ and
  $D_{vert}' = -(x-s)/t\del_x = -[(x -(s+tw_i))/t + w_i]\del_x$.
  Thus $D_{vert} - D_{vert}' = (x-z_i)/t\partial_x$, hence $a_1 = 1/t$.
\item For the remaining sections $a_1 = 0$.
\end{itemize}

As before, it follows from the above computations that the residue
along this divisor is the endomorphism of $\vl|_B$ which acts by
multiplication by 
\begin{equation}
\label{eq:kzres4}
\frac{c(\mu)  - \sum_{i \in S} c(\l_i)}{2(\ell + h^{\vee})}
\end{equation}
on the summand $\V_{\ovl'{\mu}} \otimes \V_{\ovl''{\mu^*}}$ for each
$\mu \in P_{\ell}$.

\subsection{} \label{sec:results} 

Our formula for the Chern classes, aside from using the computations
above also uses a result of Esnault and Verdier \cite[Appendix B,
Corollary]{EV-logderham} which we recall here for the reader's
convenience:

Let $X$ be a smooth projective variety, $D = \cup_i D_i$ a divisor
with simple normal crossings on $X$, $V$ a vector bundle on $X$ and
$\nabla$ a connection on $U = X- D$ with logarithmic singularities.
\begin{prop}[Esnault, Verdier] \label{prop:ev} 
\[
N_p(V) = (-1)^p \sum_{\alpha_1  + \dots + \alpha_s = p}
{p \choose \alpha}\mr{Tr}(\Gamma_1^{\alpha_1}\circ \dots \circ \Gamma_s^{\alpha_s})
[D_1]^{\alpha_1}\cdots[D_s]^{\alpha_s}
\]
where $N_p$ denotes the $p$'th Newton polynomial in the Chern roots of
$V$, $[D_i]$ denotes the class of $D_i$ in Hodge cohomology and
$\Gamma_i$ is the endomorphism of $V|_{D_i}$ given by the residue of
$\nabla$ along $D_i$.
\end{prop}

\smallskip

The KZ connection on $\vl$ has logarithmic singularities, so using the
computations of the residues of the KZ connection along the boundary
divisors of $\mnb$  and Proposition \ref{prop:ev} we get
the following (implicit) expression for all the Chern classes of $\vl$
in Hodge cohomology or, equivalently, the rational Chow groups of
$\mnb$. The result is
\begin{thm} \label{thm:chern} 
Let $\mg{g}$ be a simple Lie algebra, $\ell \geq 0$ an integer and
$\ovl = (\l_1,\l_2,\dots,\l_n) \in {P_{\ell}}^n$. Then
\begin{equation}
\label{eq:chern}
N_p(\vl) = (-1)^p \sum_{\alpha_1  + \dots + \alpha_s = p}
{p \choose \alpha}\mr{Tr}(\Gamma_1^{\alpha_1}\circ \dots \circ \Gamma_s^{\alpha_s})
[B_1]^{\alpha_1}\cdots[B_s]^{\alpha_s}
\end{equation}
in $\mr{CH}^p(\mnb)_{\Q}$. Here $N_p$ denotes the $p$'th Newton class, the
$B_i$, $i=1,\dots,s$, are the irreducible components of $\mnb \bs
\mn$ and $\Gamma_i$ denotes the residue of the KZ connection
along $B_i$ given by one of \eqref{eq:kzres1}, \eqref{eq:kzres2},
\eqref{eq:kzres3}, and \eqref{eq:kzres4}.
\end{thm} 

\begin{rems} \mbox{}
\begin{enumerate} 
\item Keel \cite{keel-m0nbar} has determined the Chow ring
  of $\mnb$, so all intersections involved in \eqref{eq:chern} may be
  computed explicitly.
\item To compute the traces one needs to know the ranks of the bundles
  of conformal blocks. These are given in closed form by the Verlinde
  formula for the classical groups and $G_2$ \cite{sorger-verlinde} or
  can be derived inductively from the $3$-point ranks using
  \cite[Corollary 3.5.2]{ueno-conformal}.
\item The traces appearing in \eqref{eq:chern} are rational numbers
  but not, in general, integers.
\end{enumerate}
\end{rems}

Since the KZ connection depends on the choice of coordinates, so do the
residues, hence also the representing cycle for $c_1(\vl) = N_1(\vl)$ in
Theorem \ref{thm:chern}. However, by averaging over all choices we
obtain a canonical representative:

\begin{cor} \label{cor:cherncan} 

Let $\mg{g}$ be a simple Lie algebra, $\ell \geq 0$ an integer and
$\ovl = (\l_1,\l_2,\dots,\l_n) \in {P_{\ell}}^n$. Then
\begin{multline} 
\label{eq:cherncan}
c_1(\vl) = \\
\dfrac{1}{2(\ell + h^{\vee})}\sum_{i=2}^{\lf n/2 \rf} \ep_i\Biggl \{
\sum_{\substack{A \subset \{1,2,\dots,n\} \\|A| = i}} \Biggl \{
\dfrac{r_{\ovl}}{(n-1)(n-2)} \Bigl \{
(n-i)(n-i-1)\sum_{a \in A}c(\l_a) + i(i-1) \sum_{a' \in A^c} c(\l_{a'}) \Bigr \} \\
- \left \{ \sum_{\mu \in P_{\ell}} c(\mu) \cdot r_{\ovl_{A,\mu}} \cdot
r_{\ovl_{A^c,\mu^*} } \right \} \Biggr \} \cdot [D_{A,A^c}] \Biggr \}
\end{multline}
in $\pic(\mnb)_{\Q}$, where $D_{A,A^c}$ is the irreducible boundary
divisor corresponding to the partition $\{1,2,\dots,n\} = A \cup A^c$
and $\ep_i = 1/2$ if $i = n/2 $ and $1$ otherwise.
\end{cor}

\begin{proof}
  The choices involved in \S \ref{subsec:kzres} are the labelling of
  the last three points as $p_0,p_1$ and $p_{\infty}$. We consider all
  $n(n-1)(n-2)$ ways of choosing these labels and compute the cycle
  obtained by averaging the coefficients of $D_{A,A^c}$ for each
  possible choice.

  We consider the four types of boundary divisors considered in \S
  \ref{subsec:kzres} and consider the coefficient of $c(\l_a)$ for $a
  \in A$ coming from each of these divisor types. Let $i = |A|$.
  \begin{enumerate}
  \item $p_0 \in A$ and $p_1, p_{\infty} \in A^c$. There are
    $i(n-i)(n-i-1)$ such cases and from \eqref{eq:kzres1} each one of
    these gives a contribution of $-1$, for a total contribution of
    $-i(n-i)(n-i-1)$.
  \item $p_1 \in A$ and $p_0, p_{\infty} \in A^c$. There are again
    $i(n-i)(n-i-1)$ such cases and from \eqref{eq:kzres2} each one of
    these gives a contribution of $-1$, for a total contribution of
    $-i(n-i)(n-i-1)$.
  \item $p_{\infty} \in A$ and $p_0,p_1 \in A^c$.  There are
    $i(n-i)(n-i-1)$ such cases.  From \eqref{eq:kzres3} it follows
    that if $a = p_{\infty}$ then we get a contribution of $-1$ and
    otherwise we get a coefficient of $1$ so the total contribution is
    $(i-2)(n-i)(n-i-1)$.
  \item $p_0, p_1, p_{\infty} \in A^c$. There are
    $(n-i)(n-i-1)(n-i-2)$ such cases and from \eqref{eq:kzres4} it
    follows that each gives a contribution of $-1$ for a total
    contribution of $-(n-i)(n-i-1)(n-i-2)$.
  \end{enumerate}

  Summing all these we get that the coefficient of $c(\l_a)$ for $a
  \in A$ is $-n (n-i)(n-i-1)$. By symmetry it follows that the
  coefficient of $c(\l_{a'})$ for $a' \in A^c$ is $-ni(i-1)$. The
  claimed formula then follows from Theorem \ref{thm:chern}.
\end{proof}

\bigskip

Specialising Corollary \ref{cor:cherncan} to the case $n=4$, we get
the following formula which we state here for ease of reference later:
\begin{cor} 
\label{cor:deg4}
Let $\mg{g}$ be a simple lie algebra, $\ell \geq
0$ an integer and $\ovl = (\l_1, \l_2, \l_3,\l_4) \in
{P_{\ell}}^4$. Then 
\begin{multline}
\label{eq:mainformula}
\deg(\vl) = \frac{1}{2(\ell + h^{\vee})} \times \biggl\{ \Bigr \{
r_{\ovl} \sum_{i=1}^4 c(\l_i)
\Bigr \} \ -  \\
\Bigl \{ \sum_{\l \in P_{\ell}} c(\l)\bigl \{( r_{(\l_1, \l_2, \l)}
\cdot r_{(\l_3, \l_4, \l^*)} + r_{(\l_1, \l_3, \l)}\cdot r_{(\l_2,
  \l_4, \l^*)} + r_{(\l_1, \l_4, \l)}\cdot r_{(\l_2, \l_3, \l^*)}
\bigr \}
\Bigr \} \biggr \} \ .\\
\end{multline}
\end{cor}

Inserting \eqref{eq:mainformula} into Proposition \ref{prop:deg}, one
obtains a formula for $\deg(\vl|_F)$ for any vital curve $F$. Since
the vital curves generate $\mr{CH}_1(\mnb)$ this gives a dual
expression for $c_1(\V_{\ovl})$ which will be useful to us
below. Similar expressions can in principle also be obtained for the
other Chern classes.

\bigskip

For $2 \leq i \leq n/2$, let $D_i := \ep_i\sum_{A} D_{A,A^c}$ where
the sum is over all $A \subset \{1,2,\dots,n\}$ with $|A| = i$ and
$D_{A,A^c}$, $\ep_i$ are as above.  For $\ovl \in {P_{\ell}}^n$ and $\mu
\in P_{\ell}$, let $\ovl_{A,\mu}$ be the $i+1$-tuple
$(\l_{a_1},\dots,\l_{a_i},\mu)$ where $A = \{a_1,\dots,a_i\}$.

\begin{cor}  
  For the action of the symmetric group $S_n$ on $\mnb$ by permutation
  of the marked points we have
\begin{multline}
\label{eq:symm}
\sum_{\sigma \in S_n} \sigma^*(c_1(\V_{\ovl})) = \frac{1}{2(\ell +
  h^{\vee})} \sum_{2 \leq i \leq n/2} i!(n-i)!\Biggl \{ \biggl \{
\Bigl ( \tbinom{n-3}{i-1} + \tbinom{n-3}{n-i -1} \Bigr ) \cdot
r_{\ovl}\sum_{j = 1}^n c(\l_j) \biggr \} \\
- \biggl \{\sum_{|A| = i} \sum_{\mu \in P_{\ell}} c(\mu) \cdot
r_{\ovl_{A,\mu}} \cdot r_{\ovl_{A^c,\mu^*} }\biggr \} \Biggr \} \cdot
[D_i]
\end{multline}
\end{cor}

\begin{proof}
This follows from Theorem \ref{thm:chern} and a simple counting argument.
\end{proof}

\begin{rem}
  The set $\{[D_i]\}_{1 \leq i \leq n/2}$ is a basis of
  $\pic(\mnb)_{\Q}^{S_n} \cong \pic(\mnb/{S_n})_\Q$, so the RHS of
  \eqref{eq:symm} is independent of all choices. Keel and McKernan
  \cite[Theorem 1.3]{keel-mckernan} have proved that the $[D_i]$
  generate the cone of effective divisors of $\mnb/{S_n}$; since
  $c_1(\vl)$ is always nef, it follows from \emph{op.~cit.}~that the
  coefficients of $[D_i]$ in \eqref{eq:symm} are all positive or all
  $0$. This is not the case for the coefficients of $[D_{A,A^c}]$ in
  Corollary \ref{cor:cherncan}.
\end{rem}

\section{$g=0$,  $\mg{g} = \mg{sl}_2$} \label{sec:sl2} 

In this section we consider the case $\mg{g} = \mg{sl}_2$.  We
identify $P$ with $\Z$ so that $P_+$ is identified with $\Z_{\geq
  0}$. Then for any $\l \in P_+$, $c(\l) = \l^2/2 + \l$. Furthermore,
$h^{\vee}= 2$.

\smallskip

The lemma below follows from \cite[Corollary 3.5.2]{ueno-conformal} by 
using elementary  facts about the representation theory of $\mg{sl}_2$.
\begin{lem}
\label{lem:sl2}
\mbox{ } \\
\vspace{-3mm}
\begin{enumerate}
\item For any $\ell$ and $\ovl \in {P_{\ell}}^n$, $r_{\ovl} = 0$ if
  $\sum_i \l_i$ is odd, and $\vl$ is a trivial bundle, hence has
  trivial determinant, if $\sum_i \l_i \leq 2\ell$.
\item For $\ovl = (\l_1,\l_2,\l_3)$, with $\l_1 \leq \l_2 \leq \l_3$
  and $\sum_i \l_i$ even we have
\begin{equation*}
r_{\ovl} = 
\begin{cases}
  1 & \text{if $\l_3 \leq \l_1 + \l_2$ and $\sum_i \l_i \leq 2\ell$}, \\
  0 & \text{otherwise.}
\end{cases}
\end{equation*}
\end{enumerate}
\end{lem}

\subsection{The case $n=4$}

\begin{prop}
\label{prop:n=4}
Suppose $\ovl = (\l_1,\l_2,\l_3,\l_4) \in
{P_{\ell}}^4$ with $\l_1 \leq \l_2 \leq \l_3 \leq \l_4$ and $2s
:=\sum_i \l_i$  even. Then
\[
\deg(\dl) =
\begin{cases}
\max\{0, (\ell + 1 - \l_4)(s - \ell)\} & \text{if $\l_1 + \l_4 \geq \l_2 + \l_3$}, \\
\max\{0, (\ell  + 1 + \l_1  - s)(s - \ell)\} & \text{if $\l_1 + \l_4 \leq \l_2 + \l_3$} .
\end{cases}
\]
\end{prop}

\begin{proof}
  We will use descending induction on the level $\ell$. For $\ell \geq
  s$ the degree is $0$ by Lemma \ref{lem:sl2}, so we assume the
  formula holds for some $\ell \leq s$ and prove it for $\ell - 1$.

  Consider the first case. Since $\l_4 \in P_{\ell}$ we have $\l_4
  \leq \ell$ so we may assume that $\l_4 < \ell \leq s$. In this
  range, using the factorisation formula and Lemma \ref{lem:sl2}, one
  sees that $r_{\ovl}$ decreases by $1$ each time $\ell$ is decreased
  by $1$.

  We now analyse the terms occuring in the formula
  \eqref{eq:mainformula}.  Since the rank decreases by $1$ in going
  from level $\ell$ to level $\ell - 1$ it follows from factorisation
  that for each of the terms of the type $r_{(\l_1, \l_2, \l)} \cdot
  r_{(\l_3, \l_4, \l^*)}$, $r_{(\l_1, \l_3, \l)}\cdot r_{(\l_2, \l_4,
    \l^*)}$ and $r_{(\l_1, \l_4, \l)}\cdot r_{(\l_2, \l_3, \l^*)}$
  there is exactly one $\l$ which contributes a non-zero term (which
  is actually just $1$ for $\mg{sl}_2$) at level $\ell$ but gives $0$
  at level $\ell - 1$.

  Suppose this $\l = \ell$, so $\l \notin P_{\ell -1}$. Considering
  the first type and using Lemma \ref{lem:sl2} (note that $\l = \l^*$
  here) we see that we must have $\ell \leq \l_1 + \l_2$ and $\ell +
  \l_3 + \l_4 \leq 2\ell$.  This implies that $\l_i = \ell/2$ for all
  $i$. The second type leads to the same conclusion whereas the third
  gives $\ell = \l_2 + \l_3 = \l_1 + \l_4$.

  If $\l < \ell$, so $\l \in P_{\ell -1}$, it follows from Lemma
  \ref{lem:sl2} and the inequalities among the $\l_i$ that for the
  term of the first type we must have $\l_3 + \l_4 + \l = 2\ell$, so
  $\l = 2\ell - (\l_3 + \l_4)$. Similarly, we see that we must have
  $\l = 2\ell - (\l_2 + \l_4)$ for the second type, and $\l = 2\ell -
  (\l_1 + \l_4)$ for the third.

  Note that the formulas of the previous paragraph specialise to those
  of the one before that if $\l = \ell$.

  Using the induction hypothesis, we thus see that the degree for
  level $\ell -1 $ is given by
\begin{multline}
\label{eq:m1}
2(\ell + 1)\deg(\dl) = 2(\ell + 2) (\ell + 1 - \l_4)(s - \ell)\\
+ c(2\ell -(\l_3 + \l_4)) + c(2\ell -(\l_2 + \l_4)) + c(2\ell -(\l_1 +
\l_4))\ - \sum_{i=1}^4c(\l_i).
\end{multline}
Since $c(\l) = \l^2/2 + \l$ for $\mg{sl}_2$, setting $2s' =
\sum_{i=1}^4 \l_i^2$ we see that the RHS of \eqref{eq:m1} is equal to
\begin{equation*}
\begin{split}
& 2(\ell + 2) (\ell + 1 - \l_4)(s - \ell) 
+ (6 \ell^2 - 4\ell s - 4\ell\l_4 +s'
+ 2\l_4 s) 
+ (6\ell - 2s - 2\l_4) - (s' + 2s) \\
& = 2(\ell + 2) (\ell + 1 - \l_4)(s - \ell) 
+  6 \ell^2 - 4\ell s - 4\ell\l_4 + 2\l_4 s + 6 \ell -4 s -2\l_4 \\
&= 2[ (\ell^2 + \ell - \ell\l_4 + 2\ell + 2  -2\l_4)  -2\ell + \l_4 -2]s \\
&+ 2[(\ell^2 + \ell - \ell\l_4 + 2\ell + 2  -2\l_4) - 3\ell +2\l_4 - 3)](-\ell) 
-4\l_4 \\
& = 2[(\ell + 1)(\ell - \l_4)]s  + 2[(\ell  + 1)(\ell - \l_4)  + (-\ell + \l_4 -1)](-\ell) -4\l_4\\
&=2(\ell + 1)(\ell - \l_4)(s - \ell + 1)
\end{split}
\end{equation*}
Dividing by $2(\ell + 1)$ we get the formula of the proposition for
level $\ell - 1$.

The second case is proved in an entirely analogous way so we omit the
details.

\end{proof}

\subsection{The case $\ell = 1$}

If $\ell = 1$, $P_1 = \{0,1\}$ so there are $2^n$ bundles of conformal
blocks on $\mnb$. However, of these the ones corresponding to an odd
number of $1$'s are $0$ by Lemma \ref{lem:sl2} and the ones with two
$1$'s are trivial bundles of rank $1$. Thus the maximal number of
non-trivial determinants that one can possibly get is $2^{n-1} -
\binom{n}{2} - 1$, which is the same as the rank of $\pic(\mnb)$. In
fact, one has the following
\begin{thm} 
\label{thm:basis}
  For any $n \geq 4$, the set of non-trivial determinants of conformal
  blocks of level $\ell =1$ for $\mg{sl}_2$ form a basis of
  $\pic(\mnb)_{\Q}$.
\end{thm}

\begin{proof}
We use induction on $n$.

Suppose $n=4$. Then the only possibly non-trivial determinant is for
$\ovl = (1,1,1,1)$.
From Proposition \ref{prop:n=4} we get that $\deg(\dl) = 1$, proving
the first step of the induction.

Now suppose $n>4$. Let $Q$ be the quotient of
$\pic(\mnb)_{\Q}$ by the subspace $P'$ generated by
$f_i^*(\pic(\ov{\mr{M}}_{0,n-1})_{\Q})$, $i=1,2,\dots,n$, where the
$f_i$ are forgetting maps as before.

Recall from \cite{keel-m0nbar} that for any $n\geq 4$, $\pic(\mnb)$ is
generated by the classes of the boundary divisors which are
parametrised by partitions $\{1,\dots,n\} = A \cup B$ with $|A|,|B|
\geq 2$. We denote the corresponding divisor by $D_{A,B}$, so we have
$D_{A,B} = D_{B,A}$.  For a partition $\{1,2,\dots, \wh{i},\dots,n\} =
A' \sqcup B'$ corresponding to the boundary divisor $D_{A',B'}$ on
$\ov{\mr{M}}_{0,n-1}$ (with points labelled by elements of $\{1,\dots,
\wh{i},\dots,n\}$), we have 
\begin{equation}
\label{eq:pullback}
f_i^*(D_{A',B'}) = D_{A' \cup \{i\}, B'} + D_{A', B'   \cup\{i\}} .
\end{equation}

Let $A = \{1,2,\dots,r\}$ and $B = \{r+1,r+2,\dots,n\}$ with $2 \leq r
\leq n/2$.  By switching elements of $A$ and the first $r$ elements of
$B$ in pairs using equation \eqref{eq:pullback} we get
\[
D_{\{1,2,\dots,r\},\{r+1,r+1,\dots,n\}} = D_{\{r+1,r+2,\dots,2r\},\{1,2,\dots,r,2r +1,\dots,n\}} 
\]
in $Q$. Then moving $2r +1, 2r + 2,\dots,n$ to the first set in the
new partition and using the same equation we get
\[
D_{\{r+1,r+2,\dots,2r\},\{1,2,\dots,r,2r +1,\dots,n\}} =
(-1)^{n-2r}D_{\{r+1,r+1,\dots,n\},\{1,2,\dots,r\}} = (-1)^n
D_{\{1,2,\dots,r\},\{r+1,r+1,\dots,n\}}
\]
in $Q$. If $n$ is odd, it follows that $D_{A,B} = 0$ in $Q$. Since the
symmetric group $S_n$ acts on $Q$, it follows that $Q=0$ in this case.
If $n$ is even, a similar argument shows that $Q$ has rank $1$,
generated by $D_{\{1,2,\dots,n/2\},\{n/2+1,n/2 +2,\dots,n\}}$.

As we observed earlier, if some $\l_i = 0$ then the bundle of
conformal blocks is pulled back via $f_i$, hence its determinant is in
$P'$. To complete the proof it remains to show that
$\D_{(1,1,\dots,1)}$ generates $Q$ if $n$ is even.

Suppose not, so $\D_{(1,1,\dots,1)}$ lies in $P'$. For even $r$,
$4\leq r \leq n$, let $\D_r$ be the sum of all the determinants of
conformal blocks of level $1$ with $r$ of the $\l_i$ equal to
$1$. Since $\D_n$ is preserved by the action of $S_n$, it follows
by averaging that we must have a linear relation
\[
\sum_{i=2}^{n/2} a_i \D_{2i} = 0
\]
with $a_i \in \Q$.

Let $\iota:\ov{\mr{M}}_{0,n-1} \to \mnb$ be the morphism corresponding
to attaching a $3$-pointed $\P^1$ to the last marked point. Let $\ovl'
= (\l_1,\l_2,\dots, \l_{n-2}, 0)$ with $\l_i \in \{0,1\}$ and $\sum_i
\l_i = 2r \geq 4$.  The coefficient of $\D_{\ovl'}$ in
$\iota^*(\sum_{i=2}^{n/2} a_i \D_{2i})$, which is well defined by
induction, must be zero. By the factorisation formula this coefficient
is equal to $a_{i} + a_{i+1}$ so we must have
\[
a_{i} + a_{i+1}=0 \ , 2\leq i \leq n/2 -1 .
\]
Now let $F$ be the vital curve in $ \mnb$ corresponding to the
partition $\{1,2,\dots,n\} = \{1\} \cup \{2\} \cup \{3\} \cup
\{4,5,\dots,n\}$. Using the computation in the $n=4$ case and
Proposition \ref{prop:deg} we see that for any $\ovl =
(\l_1,\l_2,\dots, \l_{n}) $ with $\sum_i \l_i = r$ even, $4\leq r \leq
n$
\[
\dl\cdot F = 
\begin{cases}
1 & \text{if $\l_1 = \l_2= \l_3 = 1$}, \\
0 & \text{otherwise}.
\end{cases}
\]
It follows that $\D_r \cdot F = {n-3 \choose r-3}$ for $r$ even,
$4\leq r \leq n$.  Since $\sum_{i=2}^{n/2} a_i \D_{2i}\cdot F =
0$, by putting together the two relations obtained so far we get
\[
f(n) := {n-3 \choose n-3} - {n-3 \choose n-5} + {n-3 \choose n-7}\dots
+ (-1)^{n/2} {n-3 \choose 1} = 0 .
 \]
 However, by the binomial theorem, $2f(n) = (1+\sqrt{-1})^n + (1 -
 \sqrt{-1})^n$, which is clearly non-zero. This contradiction
 completes the proof of the theorem.
\end{proof}

\begin{rem}
  The basis of Theorem \ref{thm:basis} has several nice properties
  which are clear from the construction: All the elements are nef line
  bundles and the basis is preserved by the action of the symmetric
  group. In fact, the basis is compatible with all natural morphisms
  among the $\mnb$'s including the forgetting and gluing morphisms.
  However, even though our basis is contained in $\pic(\mnb)$, it does
  not form an integral basis for $n>4$ as may be seen by explicit
  computation.
\end{rem}

\subsection{The critical level and GIT quotients}

For $\ovl = (\l_1,\dots,\l_n)$ we call $\ell = (\sum_i \l_i/2) -1$ the
critical level associate to $\ovl$. This is the largest possible value
for the level so that the bundle $\vl$ is not trivial. In this section
we identify the line bundles $\dl$ for the critical level and with all
$\l_i > 0$ with pullbacks of ample line bundles on the GIT quotients
$(\P^1)^n \sslash SL_2$, where the polarisation on $(\P^1)^n$ is given
by the line bundle $\mc{O}(\l_1) \boxtimes \cdots \boxtimes
\mc{O}(\l_n)$. We will do this by comparing the degrees of the $\dl$
on vital curves with the degrees of the GIT bundles for which a
formula has been given by Alexeev--Swinarski \cite{alexeev-swinarski}.

For $\ovl$, with none of the $\l_i = 0$, let $p_{\ovl}: \mnb \to
(\P^1)^n\sslash SL_2$ denote the morphism constructed by Kapranov
\cite{kapranov-chow}, where the GIT quotient is constructed using the
polarisation corresponding to $\ovl$.
\begin{thm}
\label{thm:crit}
For the critical level $\ell = \sum_i \l_i/2 -
1$, $\dl$ is a multiple of the pullback by $p_{\ovl}$ of the canonical
ample line bundle on $ (\P^1)^n\sslash SL_2$. Furthermore, $\vl$ is
the pullback of a vector bundle on $ (\P^1)^n\sslash SL_2$.
\end{thm}
\begin{proof}
  We first show that the degree of $\dl$ on any vital curve $F$ is a
  fixed multiple of the degree of the GIT bundle on this curve. This
  suffices for the first part since $\pic(\mnb)$ is finitely
  generated and torsion free.

  Let $F$ correspond to the partition $\{1,\dots,n\} = \sqcup_{j=1}^4
  N_j$.  For $j =1,\dots,4$, let $\nu_j = \sum_{k \in N_j} \l_k$. Let
  $\nu_{max} = \max_j \{\nu_j\}$ and $\nu_{min} = \min_j
  \{\nu_j\}$. From Proposition \ref{prop:deg} we have
  \[
  \deg(\dl)|_F =\sum_{\ov{\mu} \in {P_{\ell}}^4}
  \deg(\mb{V}_{\ov{\mu}})\ \prod_{j=1}^4 r_{\ovl_{\mu_j^*}} \ .
  \]
  Note that $\mu_j = \mu_j^*$ since $\mg{g} = \mg{sl}_2$. To get a
  non-zero summand, each of its factors must be non-zero. Considering
  the ranks, we see from Lemma \ref{lem:sl2} that this implies that
  for each non-zero summand we must have $\mu_j \leq \nu_j$ for all
  $k$. Now considering the term $\deg(\mb{V}_{\ov{\mu}})$ and applying
  Lemma \ref{lem:sl2} again, we see that we must have $\sum_j \mu_j
  \geq 2\ell + 2$. Since $2\ell = \sum_i \l_i -2 = \sum_j \nu_j - 2$,
  it follows that all the inequalities must be equalities. Thus there
  is only one possibly non-zero summand corresponding to $\mu_j =
  \nu_j$, $j=1,\dots,4$. In this summand, since $\ell \geq 2\nu_j$ for
  all $j$, by applying the factorisation formula one sees that
  $r_{{\ovl}^j} = 1$ for all $j$.
  
  It follows from Proposition \ref{prop:n=4} that
  \begin{equation} \label{eq:degcrit}   
  \deg(\dl)|_F =
  \begin{cases}
   0 & \text{if $\nu_{max} \geq \ell + 1$}, \\
   \ell + 1 -\nu_{max} & \text{if $\nu_{max} \leq \ell + 1$ and $\nu_{max} + \nu_{min} \geq \ell + 1$}, \\
   \nu_{min} & \text{if $\nu_{max} \leq \ell + 1$ and $\nu_{max} + \nu_{min} \leq \ell + 1$}.
 \end{cases}
\end{equation}
This is exactly the same, upto scaling, as the formula of
Alexeev--Swinarski \cite[Lemma 2.2]{alexeev-swinarski}.  Since the
scaling factor is independent of the specific vital curve $F$, the
statement about $\dl$ follows.

Since $\vl$ is generated by its global sections (Lemma \ref{lem:gen}),
there exists a morphism $f_{\ovl}:\mnb \to \mr{Gr}_{\ovl}$, where
$\mr{Gr}_{\ovl}$ is a grassmannian, such that $\vl$ is isomorphic to
$f_{\ovl}^*$ of the tautological vector bundle on $\mr{Gr}_{\ovl}$.
Since $\dl$ is the determinant of $\vl$ and the GIT quotient
$(\P^1)^n\sslash SL_2$ is normal, it follows that we have $f_{\ovl} =
q_{\ovl} \circ p_{\ovl}$ for some morphism $q_{\ovl}: (\P^1)^n \sslash
SL_2 \to \mr{Gr}_{\ovl}$. Thus $\vl$ is $p_{\ovl}^*$ of $q_{\ovl}^*$
of the tautological vector bundle.
\end{proof}

\subsection{Relationship with the moduli spaces of weighted pointed curves}

For any $g \geq 0$, Hassett \cite{hassett-weighted} has constructed
moduli spaces of weighted pointed stable curves. For $g = 0$, the
spaces $\ov{\mr{M}}_{0,\mc{A}}$ depend on a choice of weight data
$\mc{A} = (a_1,a_2,\dots,a_n)$ satisfying $0 < a_i \leq 1$ for all $i$
and $\sum_i a_i > 2$. There are canonical birational morphisms
$p_{\mc{A}}:\mnb \to \ov{\mr{M}}_{0,\mc{A}}$.

\begin{lem} \label{lem:wcont} 
  The morphsisms $p_{\mc{A}}$ are compositions of extremal
  contractions --- in fact smooth blowdowns --- corresponding to
  images of classes of vital curves.
\end{lem}

The following proof was communicated to us by Valery Alexeev.
\begin{proof}
  For any $S \subset \{1,2,\dots,n\}$ such that $2 < |S| < n-2$
  Hassett defines an associated wall in the space $\mc{D}_{0,n}$ of
  allowable weight data given by the equation $\sum_{i \in S} a_i =
  1$. The set of all such walls induces a decompostion of
  $\mc{D}_{0,n}$ called the coarse chamber decomposition.

  For each $i$ let $\ep_i$ be such that $ 0 < \ep_i << 1$, let $a_i' =
  a_i - \ep_i$ and $\mc{A}' = (a_1',a_2',\dots,a_n')$. Then there is a
  natural morphism $\rho_{\mc{A},\mc{A}'}:\ov{\mr{M}}_{0,\mc{A}} \to
  \ov{\mr{M}}_{0,\mc{A}'}$ such that $\rho_{\mc{A},\mc{A}'} \circ
  p_{\mc{A}} = p_{\mc{A}'}$ which is an isomorphism if the $\ep_i$ are
  sufficiently small.  Thus we may assume that the $a_i$'s are general,
  \emph{i.e.}, the line segment joining $(a_1,a_2,\dots,a_n)$ and
  $(1,1\dots,1)$ intersects the walls of the coarse chamber
  decomposition in only codimension one faces. It therefore suffices
  to show that each such crossing corresponds to the smooth blowdown
  associated to the image of a vital curve.  This follows from
  Proposition 4.5 of \cite{hassett-weighted}, the vital curve may be
  taken to correspond to a partition $\{1,2,\dots,n\} = \sqcup_{j=1}^4
  N_j$ with $ S =\cup_{j=1}^3 N_j$.
\end{proof}

Given a level $\ell$ less than the critical level, \emph{i.e.}, $\ell <
\sum_i \l_i/2 - 1$, and an allowable tuple of weights $\ovl$, we let
$a_i = \l_i/(\ell+1)$ and $\mc{A}_{\ovl} = (a_1,a_2,\dots,a_n)$. If
all $\l_i > 0$, then we have $0 < a_i \leq 1$ for all $i$ and $\sum_i
a_i > 2$, so $\mc{A}_{\ovl}$ can be taken to be weight data in the
sense of Hassett. We therefore have a moduli space
$\ov{\mr{M}}_{0,\mc{A}_{\ovl}}$ and a birational morphism
$p_{\mc{A}_{\ovl}}:\mnb \to \ov{\mr{M}}_{0,\mc{A}_{\ovl}}$ as above.

For non-critical levels, the bundles $\dl$ are often ample, so are not
pulled back from the GIT quotients. However, we have
\begin{prop} 
\label{prop:hassett}
If all $\l_i > 0$ then $\vl$ is the pullback by $p_{\mc{A}_{\ovl}}$ of
a vector bundle on $\ov{\mr{M}}_{0,\mc{A}_{\ovl}}$.
\end{prop}

\begin{proof}
  We first show that $\dl$ is the pullback by $p_{\mc{A}_{\ovl}}$ of a
  line bundle on $\ov{\mr{M}}_{0,\mc{A}_{\ovl}}$.  By Lemma \ref{lem:wcont}
  it suffices to show that $\deg(\dl)|_F = 0$ for any vital curvve $F$
  contracted by $p_{\mc{A}_{\ovl}}$. Let $F$ corresponds to the
  partition $\{1,2\dots,n\} = \sqcup_{j=1}^4 N_j$ and let
  $b_j = \sum_{k \in N_j} a_k$.  The morphism $p_{\mc{A}_{\ovl}}$
  collapses $F$ if and only if $(\sum_l b_l) - b_j < 1$ for some $j\in
  \{1,2,3,4\}$. Equivalently, setting $\nu_j =\sum_{k \in N_j} \l_k$,
  if and only if
  \begin{equation}
\label{eq:e1}
    \sum_j \nu_j  - \nu_{j'} \leq \ell   \text{ for some } j' \in
  \{1,2,3,4\} .
\end{equation} 

We now apply Proposition \ref{prop:deg}. If a tuple $\ov{\mu} \in
{P_{\ell}}^4$ is to contribute a non-zero summand, all the ranks have
to be non-zero so we must have $\mu_j \leq \nu_j$ for all
$j=1,\dots,4$. Equation \eqref{eq:e1} then implies that
\begin{equation}
  \sum_j \mu_j  - \mu_{j'} \leq \ell   \text{ for some } j' \in
  \{1,2,3,4\} .
\end{equation} 
Since $\mu_{j'} \leq \ell$ as well, we get $\sum_j \mu_j \leq 2\ell$.
It follows from Lemma \ref{lem:sl2} that
$\deg(\D_{\ov{\mu}}) = 0$.

We conclude that all the summands in the formula for $\deg(\dl|_F)$
are $0$, hence $\deg(\dl|_F) = 0$.

The statement about $\vl$ now follows in the same way as the
corresponding statement in Theorem \ref{thm:crit} since
$\ov{\mr{M}}_{0,\mc{A}_{\ovl}}$ is smooth, hence also normal.
  
\end{proof}

\begin{rem}
  It is not true in general that $\dl$ is the pullback by
  $p_{\mc{A}_{\ovl}}$ of an ample line bundle on
  $\ov{\mr{M}}_{0,\mc{A}_{\ovl}}$ : Let $n=6$, $\ell = 1$ and $\ovl =
  (1,1,1,1,1,1)$. In this case $\dl$ has degree $0$ on any vital curve
  corresponding to the partition $6 = 1 + 1 + 2 + 2$ and the
  correponding morphism is the well known birational morphism from
  $\ov{\mr{M}}_{0,6}$ to the Igusa quartic. However, the morphism
  $p_{\mc{A}_{\ovl}}:\ov{\mr{M}}_{0,6} \to
  \ov{\mr{M}}_{0,\mc{A}_{\ovl}}$ is an isomorphism since all $a_i =
  1/2$.
\end{rem}

\section{$g=0$, arbitrary $\mg{g}$  and  $\ell = 1$} 
\label{sec:level1} 

For an aribtrary simple Lie algebra $\mg{g}$ we are not able to say
much about the Chern classes of conformal blocks for arbitrary levels
$\ell$. However, in the simplest non-trivial case of level $1$ the
formulae simplify considerably and we discuss them in more detail in
this section. For the simply-laced Lie algebras the sheaves $\ovl$
have rank at most one and we get the simplest formulae in these cases.

\subsection{}
The following is the list of level $1$ representations for the simple
Lie algebras. We use the notation of Bourbaki \cite{bourbaki-roots}.
\begin{itemize}
\item $A_{\ell}$ : all fundamental weights; $\vp_i$ is dual to
  $\vp_{\ell +1 -i}$.
\item $B_{\ell}$, $\ell \geq 2$ : $\vp_1$ and $\vp_{\ell}$,
  \emph{i.e.},  the standard representation and the spin representation;
  these representations are self dual.
\item $C_{\ell}$ : all fundamental weights; these are all self dual.
\item $D_{\ell}$, $\ell \geq 3$ : $\vp_1$, $\vp_{\ell -1}$ and
  $\vp_{\ell}$, \emph{i.e.},  the standard representation and both the
  spin representations; the first representation is self dual and the
  other two are self dual if $\ell$ is even and dual to each other if
  $\ell$ is odd.
\item $E_6$ : $\vp_1$ and $\vp_6$; these representations are dual to
  each other.
\item $E_7$ : $\vp_7$; this is self dual.
\item $E_8$ : there are no level $1$ representations.
\item $F_4$ : $\vp_4$; this is self dual.
\item $G_2$ : $\vp_1$; this is self dual.
\end{itemize}

\subsection{}

Recall (see \emph{e.g.}~\cite[Corollary 3.5.2]{ueno-conformal}) that
for a simple Lie algebra $\mg{g}$, any level $\ell$ and $\lambda, \mu
\in P_{\ell}$, we have $r_{(\lambda,\mu, 0)} = 1$ if $\mu = \lambda^*$
and $0$ otherwise.  Moreover, for $\l_1,\l_2,\l_3$ of level $1$,
$\mb{V}_{(\l_1, \l_2,\l_3)}$ is the quotient of $(V_{\l_1} \otimes
V_{\l_2} \otimes V_{\l_3})_{\mg{g}}$ by the image of the subspace
$V_{\l_1}^{(1)} \otimes V_{\l_2}^{(1)} \otimes V_{\l_3}^{(1)}$, where
$V_{\l_i}^{(1)}$ is the $\mg{s}$-submodule of $V_{\l_i}$ which is the
direct sum of all the non-trivial irreducible $\mg{s}$-submodules.
Since $H_{\theta}$ has no covariants on this subspace, it follows that
the image of this subspace in $(V_{\l_1} \otimes V_{\l_2} \otimes
V_{\l_3})_{\mg{g}}$ is $\{0\}$. Therefore $\mb{V}_{(\l_1,\l_2,\l_3)} =
(V_{\l_1} \otimes V_{\l_2} \otimes V_{\l_3})_{\mg{g}}$.

\subsubsection{$A_{\ell}$}

In this case given $\l_1$, $\l_2$ of level $1$, there exists a unique
$\l_3$ of level $1$ such that $r_{(\l_1,\l_2,\l_3)} = 1$. By the
factorisation formula, it follows that for $\l_1,\l_2,\dots,\l_{n-1}$
of level $1$, there exists a unique $\l_n$ of level $1$ such that for
$\ovl = (\l_1,\l_2,\dots,\l_n)$, $r_{\ovl} \neq 0$. Moreover, it then
follows that $r_{\ovl} = 1$. Because of this there is only one
non-zero summand in the formula in Proposition \ref{prop:deg} and so
the computation of the determinants reduces to the computation of the
degree for the case $n=4$.

Note that 
\[
c(\vp_i) = (\vp_i|\vp_i) + 2(\vp_i|\rho) = i(m-i)/m  + i(m-i) = i(m-i)(m+1)/m \ .
\] 
where $m = \ell+1$.
Using this, the fact that $h^{\vee}$ for $\mg{sl}_m$ is $m$ and formula
\eqref{eq:mainformula}, one easily checks the following:
\begin{lem}
  Let $\vp_i, \vp_j, \vp_k, \vp_l$ be fundamental dominant weights of
  $\mg{sl}_m$ and suppose that $i \leq j \leq k \leq l$. For $\ovl =
  (\vp_i, \vp_j, \vp_k, \vp_l)$, we have
\[
\deg(\dl) = \begin{cases}
i & \text{if $i+j+k+l = 2m$ and $j+k \geq i + l$,}\\
m-l & \text{if $i+j+k+l = 2m$ and $j+k \leq i + l$,} \\
0 & \text{otherwise}. 
\end{cases}
\]
\end{lem}

From this and Proposition \ref{prop:deg} we immediately get the
following:
\begin{prop} \label{prop:sln1} 

Let $\ovl = (\vp_{i_1}, \vp_{i_2},
\dots,\vp_{i_n})$ with $0 \leq i_j < m$ for $j = 1,2,\dots,n$, where
$\vp_0 := 0$. Let $F$ be a vital curve in $\mnb$ corresponding to a
partition $\{1,2,\dots,n\} = \sqcup_{k=1}^4 N_k$. Let $\nu_k$ be the
representative in $\{0,1,\dots,m-1\}$ of $\sum_{j \in N_k} i_j$ modulo
$m$.  Let $\nu_{max} = \max_k \{\nu_k\}$ and $\nu_{min} = \min_k
\{\nu_k\}$. Then
\[
\deg(\dl|_F) = 
\begin{cases}
\nu_{min} & \text{if $\sum_k \nu_k = 2m$ and $\nu_{max} + \nu_{min} \leq m$,} \\
m - \nu_{max} & \text{if $\sum_k \nu_k = 2m$ and $\nu_{max} + \nu_{min} \geq m$,} \\
0 & \text{otherwise}.
\end{cases}
\]
\end{prop}

\begin{rem}
  Comparing this formula with the formula for the degrees for
  $\mg{sl}_2$ and the critical level \eqref{eq:degcrit}, we see that
  the critical level $\ell$ determinants for $\mg{sl}_2$ are a special
  case for the level $1$ determinants for $\mg{sl}_{\ell + 1}$. More
  precisely, given an $n$-tuple of non-negative integers
  $(i_1,i_2,\dots,i_n)$ such that $\sum_ji_j$ is even, we let $\ell =
  \sum_ji_j/2 -1$. Then the determinant of the bundle of conformal
  blocks on $\mnb$ associated to $\mg{sl}_2$ with weights
  $(i_1,i_2,\dots,i_n)$ and level $\ell$ is isomorphic to the bundle
  of conformal blocks (it is already of rank $1$) on $\mnb$ associated
  to $\mg{sl}_{\ell +1}$ with weights $(\vp_{i_1}, \vp_{i_2},
  \dots,\vp_{i_n})$ and level $1$.
\end{rem}

\subsubsection{$B_{\ell}$}
In this case we have $r_{(\vp_1,\vp_1,\vp_1)} =
r_{(\vp_{\ell},\vp_{\ell},\vp_{\ell})} = r_{(\vp_1,\vp_1,\vp_{\ell})}
= 0$ whereas $r_{(\vp_1,\vp_{\ell}, \vp_{\ell})} = 1$. From this and
the factorisation formula, one sees that for $\ovl \in {P_1}^n$,
$r_{\ovl} = 0$ unless $m$, the number of $\l_i$ equal to $\vp_{\ell}$,
is even. If $m=0$ then the number of $\vp_1$ must be even and then
$r_{\ovl} =1 $ whereas if $m >0$ then the rank is $2^{m/2 - 1}$.  The
possible $\ovl \in {P_1}^4$ with $\deg(\vl) > 0$ are, upto order,
$(\vp_1,\vp_1,\vp_1,\vp_1)$, $(\vp_1,\vp_1,\vp_{\ell},\vp_{\ell})$ and
$(\vp_{\ell},\vp_{\ell},\vp_{\ell},\vp_{\ell})$.

From the tables in \cite{bourbaki-roots} we see that $h^{\vee} = 2\ell
-1$, $c(\vp_1) = 2\ell$ and $c(\vp_{\ell})
= \tfrac{\ell(2\ell + 1)}{4}$.

\subsubsection{$C_{\ell}$}
In this case, the fusion for level $1$ is similar to the fusion for
$\mg{sl}_2$ at level $\ell$, \emph{i.e.}, the bijection between
$P_{1,\mg{sp}_{2\ell}}$ and $P_{\ell,\mg{sl}_2}$ which sends $\vp_i$
to $i$ preserves the $3$-point ranks. This can be seen by explicit
computation using the generalised Littlewood-Richardson rule of
Littelmann \cite[p.~42]{littelmann-characters}.  By factorisation, it
follows that all $n$-point ranks are preserved. For the degree in the
case $n=4$, we have:
\begin{prop} \label{prop:spn=4} 
Suppose $n=4$, $i \leq j \leq k \leq l$ and $2s := i+j+k+l$
is even.  Then for any rank $\ell \geq l$
and $\ovl = (\vp_i,\vp_j,\vp_k,\vp_l)$ we have
\[
\deg(\dl) =
\begin{cases}
\max\{0, (\ell + 1 - l)(2s - \ell - l)/2\} & \text{if $i + l \geq j + k$ and $\ell \leq s$}, \\
(s + 1 -l)(s - l)/2  & \text{if $i + l \geq j + k$ and $\ell \geq s$}, \\
\max\{0, (\ell  + 1 + i  - s)(i - \ell + s)/2\} & \text{if $i+k \leq j+l$ and $\ell \leq s$}, \\
i(i+1)/2 & \text{if $i+k \leq j+l$ and $\ell \geq s$}.
\end{cases}
\]
\end{prop}
\begin{proof}
  From the tables in \cite{bourbaki-roots}, one computes that for
  $\mg{sp}_{2\ell}$, $h^{\vee} = \ell + 1$ and $c_{\ell}(\vp_i) =
  i(\ell - i/2 + 1)$, where we use a subscript for the Casimir action
  to emphasize that it depends on $\ell$. Using this, one may prove
  the proposition in a similar way to the proof of Proposition
  \ref{prop:n=4} so we only describe the changes that need to be made.

  Consider the first case. Since the degree is not necessarily zero
  for large $\ell$ we use (increasing) induction beginning with the
  base case $\ell = l$.  It follows from factorisation that for each
  of the terms of the type $r_{(\vp_i, \vp_j, \l)} \cdot r_{(\vp_k,
    \vp_l, \l^*)}$, $r_{(\vp_i, \vp_k, \l)}\cdot r_{(\vp_j, \vp_l,
    \l^*)}$ and $r_{(\vp_i, \vp_l, \l)}\cdot r_{(\vp_j, \vp_k, \l^*)}$
  occuring in the formula of Corollary \ref{cor:deg4}, there is
  exactly one $\l$ which contributes a non-zero term (which is
  actually just $1$); in particular the rank is $1$. These $\l$ are
  seen to be, in order, $\vp_{l-k}$, $\vp_{l-j}$ and $\vp_{l-i}$.  The
  degree is therefore given by
\[
\frac{1}{2(l + 2)}\{c_{\ell}(\vp_i) + c_{\ell}(\vp_j) + c_{\ell}(\vp_k) + c_{\ell}(\vp_l)
- \{ c_{\ell}(\vp_{l-i}) + c_{\ell}(\vp_{l-j}) + c_{\ell}(\vp_{l-k})\} \} .
\]
Substituting in the values of the $c_{\ell}$, simple algebraic
manipulations show that this is equal to $s-l$, proving the base
case. As seen before for the case of $\mg{sl}_2$, the rank increases
by $1$ each time the level increases by $1$ until $\ell = s$, so we
may prove the formula for all $\ell$ such that $l \leq \ell \leq s$ by
induction as before.

We now consider the second case. It follows from the first case that
the formula holds for $\ell = s$, so we use induction to prove it for
larger $\ell$. For all such $\ell$, the rank is constant and all the
terms in \eqref{eq:mainformula} for varying $\ell$ are the same except
for the fact that the Casimir actions depend on $\ell$. More
precisely, we have
\[
\deg_{\ell}(\vl) = \frac{1}{2(\ell + 2)}
\{c_{\ell}(\vp_i) + c_{\ell}(\vp_j) + c_{\ell}(\vp_k) + c_{\ell}(\vp_l)
- \sum_{\l \in P_1} c_{\ell}(\l)\cdot a_{\l}\}
\]
where $a_{\l} \in \{0,1,2,3\}$ is independent of $\ell$ in this
range. It follows from this and the formula for $c_{\ell}$ that
$\lim_{\ell \to \infty} \deg_{\ell}(\vl)$ exists. Since
$\deg_{\ell}(\vl)$ is always an integer it follows that
$\deg_{\ell}(\vl)$ is constant for $\ell \gg 0$.

The formula also shows that
\begin{align*}
&2(\ell + 2)(\ell+3)(\deg_{\ell}(\vl) - \deg_{\ell + 1}(\vl))\\
= &(\ell + 3)\{c_{\ell}(\vp_i) + c_{\ell}(\vp_j) + c_{\ell}(\vp_k) + c_{\ell}(\vp_l)
- \sum_{\l \in P_1} c_{\ell}(\l)\cdot a_{\l}\} \\
- &(\ell + 2)\{c_{\ell + 1}(\vp_i) + c_{\ell + 1}(\vp_j) + c_{\ell + 1}(\vp_k) + c_{\ell + 1}(\vp_l)
- \sum_{\l \in P_1} c_{\ell + 1}(\l)\cdot a_{\l}\} \\
= &(\ell + 3)\{c_{\ell}(\vp_i) + c_{\ell}(\vp_j) + c_{\ell}(\vp_k) + c_{\ell}(\vp_l)
- \sum_{\l \in P_1} c_{\ell}(\l)\cdot a_{\l}\} \\
- &(\ell + 3)\{c_{\ell + 1}(\vp_i) + c_{\ell + 1}(\vp_j) + c_{\ell + 1}(\vp_k) + c_{\ell + 1}(\vp_l)
- \sum_{\l \in P_1} c_{\ell + 1}(\l)\cdot a_{\l}\} \\
+& \{c_{\ell + 1}(\vp_i) + c_{\ell + 1}(\vp_j) + c_{\ell + 1}(\vp_k) + c_{\ell + 1}(\vp_l)
- \sum_{\l \in P_1} c_{\ell + 1}(\l)\cdot a_{\l}\} \\
= & -(\ell + 3)\{i + j + k + l - \sum_{\l \in P_1}\l\cdot a_{\l} \}\\
+& \{c_{\ell + 1}(\vp_i) + c_{\ell + 1}(\vp_j) + c_{\ell + 1}(\vp_k) + c_{\ell + 1}(\vp_l)
- \sum_{\l \in P_1} c_{\ell + 1}(\l)\cdot a_{\l}\} \ .\\
\end{align*}
For $\ell \gg 0$, we have seen that this is $0$ so $\deg_{\ell}(\vl) =
\{i + j + k + l - \sum_{\l \in P_1}\l\cdot a_{\l} \}/2$ for such
$\ell$.  If the formula is known for level $\ell + 1$ with $\ell \geq
s$, then the above equations imply that it also holds for level
$\ell$. This implies $\deg_{\ell}(\vl)$ is constant for all $\ell \geq
s$, hence is equal to $\deg_s(\vl)$, completing the proof in this
case.

The remaing two cases are derived in a very similar way so we omit the
details.
\end{proof}

Comparing Proposition \ref{prop:spn=4} with Proposition
\ref{prop:n=4}, we see that the degree for $\mg{sp}_{2\ell}$ is always
greater than or equal to the corresponding degree for $\mg{sl}_2$.  It
follows that for any $n \geq 4$ and $\ovl \in {P_{1,
    \mg{sp}_{2\ell}}}^n \leftrightarrow {P_{\ell, \mg{sl}_2}}^n$, the line bundle
$\dl^{\mg{sp}_{2\ell}} \otimes (\dl^{\mg{sl}_{2}})^{-1}$ has
non-negative degree on every vital curve.
\emph{However, we do not know
  if this line bundle is nef.}

\begin{rem}
  It seems possible that similar results might hold more generally,
  with conformal blocks of $\mg{sp}_{2r}$ at level $\ell$ being ``more
  positive'' than the corresponding conformal blocks of
  $\mg{sp}_{2\ell}$ at level $r$ if $r \geq \ell$. Here one should use
  the bijection between $P_{\ell,\mg{sp}_{2r}}$ and
  $P_{r,\mg{sp}_{2\ell}}$ given by replacing a Young tableau with its
  transpose. Note that this is not the same as the bijection used by
  Abe in \cite{abe-strange} in his formulation of ``strange duality''.
\end{rem}

\subsubsection{$D_{\ell}$}
Let $R \subset P$ denote the root lattice. The quotient map $P \to
P/R$ identifies $P_1$ with $P/R$ giving the former the structure of
an abelian group whose operation we denote by $\odot$ .  It is known,
see for example \cite{abe-strange}, that for $\ovl = (\l_1,\l_2,\l_3)
\in {P_1}^3$, $r_{\ovl} = 1$ iff $\l_1 \odot \l_2 \odot \l_3 = 0$ and
$r_{\ovl} = 0$ otherwise.  It follows from the the factorisation
formula that for general $\ovl \in {P_1}^n$, $\vl$ is non-zero iff
$\bigodot_i \l_i = 0$, in which case it is always a line bundle.

Let $F$ be a vital curve in $\mnb$ corresponding to a partition
$\{1,2,\dots,n\} = \sqcup_{k=1}^4 N_k$ and for each $k$, let $\nu_k =
\bigodot_{j \in N_k} \l_j$ and $\ov{\nu} :=(\nu_1,\nu_2,\nu_3,\nu_4)
$.  Then by Proposition \ref{prop:deg}, $\deg(\dl|_F) =
\deg(\D_{\ov{\nu}})$.  If any $\nu_i = 0$, then
$\deg(\D_{\ov{\nu}}) = 0$ and for $\D_{\ov{\nu}}$ to be
nontrivial we must also have $\sum_{k=1}^4 \nu_k \in R$. Furthermore,
\begin{itemize}
\item If $\ell$ is even it follows that we must have, upto ordering,
  $\nu_1 = \nu_2 \neq 0$ and $\nu_3 = \nu_4 \neq 0$ in order to get a
  non-zero degree. In this case, it then follows from Corollary
  \ref{cor:deg4} that
\begin{equation} \label{eqn:deven}
\deg(\D_{\ov{\nu}}) 
= \begin{cases} 
\frac{\ell}{2} & \text{if all $\nu_k = \nu \in \{\vp_{\ell-1},\vp_{\ell}\}$,} \\
2 & \text{if all $\nu_k = \vp_1$,} \\
\frac{\ell  -2}{2} & \text{if all $\nu_k \in  \{\vp_{\ell-1},\vp_{\ell}\}$ but not all equal,}\\
1 & \text{otherwise.}
\end{cases}
\end{equation}
\item If $\ell$ is odd it follows that we must have all $\nu_k$ equal
  or, upto order, $\ov{\nu} = (\vp_1,\vp_1,\vp_{\ell - 1},
  \vp_{\ell})$ or
  $(\vp_{\ell-1},\vp_{\ell-1},\vp_{\ell},\vp_{\ell})$. Applying
  Corollary \ref{cor:deg4} again we see that
\begin{equation} \label{eqn:dodd}
\deg(\D_{\ov{\nu}}) 
= \begin{cases}
\frac{\ell -3 }{2} & \text{if all $\nu_k = \nu \in \{\vp_{\ell - 1}, \vp_{\ell}\}$,} \\
2  & \text{if all $\nu_k = \vp_1$,} \\
\frac{\ell - 1}{2}  & \text{if all $\nu_k \in \{\vp_{\ell - 1}, \vp_{\ell}\}$ but not all equal,} \\
1 & \text{otherwise.}
\end{cases}
\end{equation}
\end{itemize}
From the tables in \cite{bourbaki-roots} we have used that $h^{\vee} =
2\ell -2$ and one also calculates that $c(\vp_1) = 2\ell -1$,
$c(\vp_{\ell-1}) = c(\vp_{\ell}) =
\tfrac{\ell(2\ell - 1)}{4}$. 

For $n>4$, one easily sees from the above that all the $\dl$ lie on
the boundary of the nef cone.  As $\ell$ varies, keeping its parity the
same, the non-trivial conformal blocks are indexed by the same data so
we identify all the sets $P_{1,\mg{so}_{2\ell}}$ for $\ell$ even with
the set $P_{1,even} := \{0,\vp_1,\vp_{e -1}, \vp_e\}$ with $e$ a formal
symbol and similarly for $\ell$ odd, with $e$ replaced by $o$. 
The set of vital curves $F$ on which the degree is $0$ is
preserved by this indexing but the bundles themselves are not.
We will use $\ell$ as a superscript in order to specify the
level.

\begin{prop} \label{prop:fgend} 
  For any integer integer $n\geq 4$, the closed subcone of
  $N_1(\mnb)_{\R}$ generated by $c_1(\D_{\ovl}^{\ell})$ for all $\ovl
  \in P_{1,even}^n$, $\ell \geq 4$ and even, or $\ovl \in P_{1,odd}^n$,
  $\ell \geq 3$ and odd, is a finitely generated rational polyhedral
  cone.
\end{prop}

\begin{proof}
  It follows from \eqref{eqn:deven} and \eqref{eqn:dodd} and the
  preceding discussion that for any $\ovl \in
  P_{1,even}^n$ or $P_{1,odd}^n$, we may divide $\D_{\ovl}^{\ell}$ by $\ell/2$
  and take the limit as $\ell \to \infty$ (while keeping its parity
  fixed) to get $\Q$-line bundles on $\mnb$ which we denote by
  $\D_{\ovl}^{even}$ and $\D_{\ovl}^{odd}$. We have
\begin{itemize}
\item For $\ov{\nu} \in P_{1,even}^4$  such that upto ordering $\nu_1 = \nu_2
  \neq 0$ and $\nu_3 = \nu_4 \neq 0$. Then
\[
\deg(\D_{\ov{\nu}}^{even}) 
= \begin{cases} 
1 & \text{if all $\nu_k  \in \{\vp_{e-1},\vp_{e}\}$,} \\
0 & \text{otherwise.}
\end{cases}
\]
\item For $\ov{\nu} \in P_{1,odd}^4$  such that all $\nu_k$ are equal or, upto
  order, $\ov{\nu} = (\vp_1,\vp_1,\vp_{o-1}, \vp_{o})$ or
  $(\vp_{o-1},\vp_{o-1},\vp_{o},\vp_{o})$. Then
\[
\deg(\D_{\ov{\nu}}^{odd}) 
= \begin{cases}
1 & \text{if all $\nu_k  \in \{\vp_{o - 1}, \vp_{o}\}$,} \\
0 & \text{otherwise.}
\end{cases}
\]
\end{itemize}
It is then immediate that for $\ell \geq 4$ and even we have
$\D_{\ovl}^{\ell} = \D_{\ovl}^4 + \tfrac{\ell -4}{2}
\cdot\D_{\ovl}^{\infty}$ and for $\ell \geq 3$ and odd we have
$\D_{\ovl}^{\ell} = \D_{\ovl}^3 + \tfrac{\ell -3}{2} \cdot
\D_{\ovl}^{\infty}$. It follows that the closed cone generated by all
$c_1(\D_{\ovl}^{\ell})$ is generated by the first Chern classes of
$\D_{\ovl}^{\ell}$, $\ell = 3,4$, $\D_{\ovl}^{even}$,
$\D_{\ovl}^{odd}$ and there are only finitely many choices for $\ovl$
for a fixed $n$.
\end{proof}

\begin{rem}
  We do not know whether $\D_{\ovl}^{even}$ and $\D_{\ovl}^{odd}$ are
  semiample. Note that since their degree on any vital curve is $0$ or
  $1$, in particular an integer, they are actually line bundles and
  not just $\Q$-line bundles. It would be interesting to have a more
  geometric description of these bundles.
\end{rem}

\subsubsection{$E_6$}
In this case\footnote{This and other such computations for exceptional
  groups mentioned below were carried out using GAP
  \cite{GAP4}.}, $r_{(\l,\l,\l)} = 1$ for $\l = \vp_1$ or $\vp_6$.
Since the representations are not self-dual, it follows that the
conformal blocks of level $1$ for $E_6$ are the same, again upto
scaling, as those for $\mg{sl}_3$.

\subsubsection{$E_7$}
In this case, $(V_{\vp_7} \otimes V_{\vp_7} \otimes V_{\vp_7})_{E_7} =
0$, so the corresponding conformal block is also trivial. It follows
that for $E_7$ the determinants of conformal blocks of level $1$ are
the same, upto scaling, as those for $\mg{sl}_2$.

\subsubsection{$E_8$}
As there are no non-trivial representations of level $\leq 1$ we do
not get any non-trivial conformal blocks. In higher genus the
situation is more interesting as we see in Corollary \ref{cor:e8}.

\subsubsection{$F_4$ and $G_2$}
In both these cases the rank of the conformal blocks for $\ovl =
(\vp_4,\vp_4,\vp_4)$ for $F_4$ and $\ovl = (\vp_1,\vp_1,\vp_1)$ for
$G_2$ is $1$. It follows that the determinants of the conformal blocks
at level $1$ for both these cases are equal (upto a global
scalar). Moreover, for $\ovl =(\vp_4,\vp_4,\vp_4,\vp_4)$ for $F_4$ and
$\ovl = (\vp_1,\vp_1,\vp_1\vp_1)$ for $G_2$, $\deg(\dl) > 0$.

For any $n$ and $\ovl = (\l,\l,\dots,\l)$, $\l = \vp_4$ or $\vp_1$ as
$\mg{g} = F_4$ or $G_2$, it follows from factorisation and the above
that $\vl$ has rank $Fib(n-1)$, where $Fib(i)$ denotes the $i$'th
Fibonacci number. Moreover, Proposition \ref{prop:deg} implies that if
a vital curve $F$ corresponds to a partition $\{1,2,\dots,n\} =
\sqcup_{k=1}^4 N_k$, then $\deg(\vl)|_F = c\prod_{k=1}^4 Fib(|N_k|)$
where $c$ is a positive constant (which can be determined). One sees
that the F-conjecture imples that $\D_{\ovl}$ is ample.

\section{The case $g >0$} \label{sec:g>0} 

As we have remarked before, for general $g$ the WZW/Hitchin connection
does not always lift to a flat connection on $\mb{V}_{g,n,\ovl}$
restricted to $\mr{M}_{g,n}$, so we cannot directly apply the same
method as in the $g=0$ case. However, we do get a flat connection on
all vector bundles induced by representations of $GL_r$, $r =
\mr{rank}(\mb{V}_{g,n,\ovl})$, which are trivial on the centre. The
Chern classes of all such bundles may, in principle, be computed as
before provided that we can compute all intersections of boundary
divisors.  Note that for such bundles the connection is canonical, so
one does not need to account for choices of coordinates and all one
needs is Proposition \ref{prop:res}.

To compute the Chern classes of $\mb{V}_{g,n,\ovl}$ itself, it
suffices to know the Chern classes of the associated bundles as above
along with $c_1(\mb{V}_{g,n,\ovl})$. By Remark \ref{rem:deg} and
Corollary \ref{cor:deg4} it follows that to compute
$c_1(\mb{V}_{g,n,\ovl})$ in general it suffices to consider the case
$(g,n) = (1,1)$.

\subsection{The case $g=1$, $n=1$}

For the rest of this section, contrary to our earlier notation, for
$\l \in P_{\ell}$ we shall denote by $\vl$ the bundle of conformal
blocks $\mb{V}_{1,1,\l}$. Our main result is:
\begin{thm} \label{thm:m11} 
  Let $\mg{g}$ be a
  simple Lie algebra, $\ell \geq 0$ an integer and $\l \in
  P_{\ell}$. Then
\[
\deg(\mb{V}_{\l}) = 
\frac{1}{2(\ell + h^{\vee})} \cdot
\Bigl \{
 \frac{r_{\l}(c_{\l} + \ell\dim(\mg{g}))}{12}
- \sum_{\mu \in P_{\ell}}c_{\mu}r_{(\l,\mu,\mu^*)}
\Bigr \} \ .
\]
\end{thm}

\begin{proof}
  In order to apply the results of \S \ref{subsec:kzconn} we must make
  explicit choices for a smooth projective curve mapping onto $\mel$,
  a coordinate for the zero section over the smooth locus and a
  bidifferential. The particular choices do not really matter; what is
  important is that choices can be made so that we get a well defined
  connection on the pullback of $\vl$ over the smooth locus.

For simplicity, we shall now work over the field of complex numbers
$\C$.  The formula for the degree we shall obtain will clearly hold
over any field of characteristic zero.

Let $\Gamma(8)$ be the principal congruence subgroup of $SL_2(\Z)$, of
level $8$ and $Y(8)$ the corresponding modular curve, \emph{i.e.}, the
quotient $\mg{H}/\Gamma(8)$, see \emph{e.g.}~\cite[Appendix A, \S
13]{silverman-aec}.  The semi-direct product $\Z^2 \ltimes \Gamma(8)$
acts on on $\C \times \mg{H}$ by
\[
\left  ((m,n), \bigl ( 
\begin{smallmatrix} 
a & b \\ c& d 
\end{smallmatrix}
\bigr ) \right ): (z,\tau) \mapsto \left (\frac{z + m\tau + n}{c\tau +
  d},\frac{a\tau + b}{c\tau + d} \right) 
\]
where $z$ (resp.~$\tau$) is the usual coordinate on $\C$
(resp.~$\mg{H}$).  The quotient is the universal family of elliptic
curves with full level $8$ structure over $Y(8)$. It extends to a
semi-stable family over the smooth compactification $X(8)$ of $Y(8)$
thus giving rise to a surjective morphism $X(8) \to \omc_{1,1}$.

Set $q = exp(\pi i \tau)$ and $w = exp(\pi i z)$, and consider the
functions $f_1$ and $f_2$ given by
\[
f_1 := \vt_{11}(2z,\tau) = \vt_{11}(w^2, q) = \sum_{m \in \Z}(-1)^mq^{(m+\tfrac{1}{2})^2}w^{4m + 2}
\]
and 
\[
f_2 := \vt_{00}(2z, \tau) = \vt_{00}(w^2,q) = \sum_{m \in \Z} q^{m^2}w^{4m} \ .
\]
The functions $\vt_{00}$ and $\vt_{11}$ are two of the classical
Jacobi theta functions as defined, for example, in
\cite{mumford-theta1}.

It follows from the transformation formulae in
\cite[p.~55]{mumford-theta1} that $f_1/f_2$ descends to a well defined
rational function on the universal family over $Y(8)$ and gives a
coordinate for the zero section. This is the coordinate we shall use.

For an elliptic curve $E = \C/ \Lambda$ there is a natural
bidifferential on $E \times E$ given by
\[
\omega := \wp_{\Lambda}(x-y)dxdy
\]
where $\wp_{\Lambda}$ denotes the Weirstrass $\wp$ function associated
to the lattice $\Lambda$. One checks that this is well defined,
i.e.~does not depend on the presentation of $E$ as $\C/\Lambda$. It
follows that this gives a family of bidifferentials associated to any
family of elliptic curves. In particular, this gives a bidifferential
on the universal family over $Y(8)$.

Having chosen coordinates and a bidifferential, we get a (flat)
connection on the pullback of $\mb{V}_{\l}$ to $Y(8)$. Since all our
data is given explicitly, we may use classical properties of $\vt$
functions and modular curves to explicitly work out all terms involved
in the discussion in Section \ref{subsec:kzconn} and Proposition
\ref{prop:ev} to compute the degree of the pullback of $\mb{V}_{\l}$
to $X(8)$ as in the $g=0$ case, hence the degee of $\mb{V}_{\l}$. This
is not difficult, however, for the sake of variety we use a somewhat
different argument which gives, as a byproduct, a simpler formula in
the case of $\mg{sl}_2$.

From the existence of the flat connection on the pullback of
$\mb{V}_{\l}$ to $Y(8)$, the discussion in Section
\ref{subsec:kzconn} and Proposition \ref{prop:ev}, it follows that
\[
\deg(\mb{V}_{\l})=
\frac{1}{2(\ell + h^{\vee})} \cdot
\Bigl \{
 r_{\l}(\alpha\ell\dim(\mg{g}) + \beta c(\l))
- \sum_{\mu \in P_{\ell}}c(\mu)r_{(\l,\mu,\mu^*)}
\Bigr \}
\]
where $\alpha$, $\beta $ are constants independent of $\mg{g}$. Here
the term $\ell\dim(\mg{g})$ comes from the bidifferential, the term
involving $c(\l)$ comes from the change of coordinate and the last
term comes from Proposition \ref{prop:res}. Note that the coefficient
of the last term is determined by the fact that there is a unique
boundary component in $\mel$.

We now show that $\alpha = \beta = 1/12$ by considering the case
$\mg{g} = \mg{sl}_2$. Let the weight $\lambda \in P_{\ell}$ corespond
to the integer $i$, so $0 \leq i \leq \ell$.  The factorisation
formula shows that the bundle of conformal blocks is $0$ unless $i$ is
even in which case the rank is equal to $\ell + 1 - i$. This is
because the $\mu$ which give rise to a non-zero summand correspond to
integers $j$ in the range from $i/2$ to $\ell - i/2$.

Since the individual local summands always have rank $1$ and
for a weight $\mu$ corresponding to an integer $j$ we have
$c(\mu) = \tfrac{j^2 + 2j}{2}$  it follows that
 the summand coming from Proposition \ref{prop:deg} 
\[
\sum_{\mu \in P_{\ell}}c(\mu)r_{(\l,\mu,\mu^*)} = \sum_{j= i/2}^{\ell -i/2} \frac{j^2 + 2j}{2} \ .
\]
The other terms are
\[
r_{\l}\alpha\ell\dim(\mg{g}) =  3 \alpha \ell (\ell + 1 -i)  
\]
and 
\[
r_{\l}\beta c(\l) = \frac{\beta  (\ell + 1 -i)(i^2 + 2i)}{2} \ .
\]

We have
\begin{multline*}
\sum_{j= i/2}^{\ell -i/2} \frac{j^2 + 2j}{2} = \sum_{j= 1}^{\ell -i/2} \frac{j^2 + 2j}{2}
-  \sum_{j= 1}^{i/2 -1} \frac{j^2 + 2j}{2} \\
= \frac{(\ell -i/2)(\ell +1 - i/2)(2\ell - i + 1)}{12}  - \frac{(i/2 -1)(i/2)(i - 1)}{12} \\
+ \frac{(\ell -i/2)(\ell + 1 -i/2)}{2} - \frac{(i/2 - 1)(i/2)}{2} \\
=  \frac{(\ell -i/2)(\ell +1 - i/2)(2\ell - i + 7)}{12}  - \frac{(i/2 -1)(i/2)(i + 5)}{12} 
\end{multline*}

If $i=0$, this is $\ell(\ell + 1)(2\ell + 7)/12$. Since $h^{\vee} = 2$
it follows that in this case we get
\[
\deg(\mb{V}_0) = \frac{1}{2(\ell + 2)}(3\alpha\ell (\ell + 1) - \ell(\ell + 1)(2\ell + 7)/12) =
\frac{1}{24(\ell + 2)}(\ell(\ell + 1)(36\alpha - 2\ell - 7) \ .
\]
Since $12\deg(\mb{V}_0)$ must be an integer for any $\ell$ it follows
that we must have $\alpha = 1/12$.

\smallskip

For an arbitrary $\l$ we therefore have
\begin{multline*}
-12\Bigl \{ r_{\l}(\alpha\ell\dim(\mg{g}) + \beta c(\l)) - \sum_{\mu \in P_{\ell}}c(\mu)r_{(\l,\mu,\mu^*)}\Bigr \} \\
= (\ell -i/2)(\ell +1 - i/2)(2\ell - i + 7)
  - (i/2 -1)(i/2)(i + 5)   - 3\ell(\ell + 1 - i) - 6\beta (\ell + 1 - i)(i^2 + 2i) \\
= (\ell -i/2)(\ell +1 - i/2)(2\ell - i + 7)
  - (i/2 -1)(i/2)(i + 5)  - (\ell + 1 -i)(6\beta i^2 + 12\beta i + 3\ell) \\
= (\ell -j )(\ell +1 -j)(2\ell -2j + 7)  -(j -1)(j)(2j + 5) + (\ell +1 -2j)(24\beta j^2 + 24\beta j + 3\ell)
\end{multline*}
where we have put $j:= i/2$ in the last line. Since
$12\deg(\mb{V}_{\l})$ is an integer, the expression, thought of as
polynomial in $\ell$, must be divisible by $\ell + 2$.  Working modulo
$\ell + 2$, the last line above is equal to
\begin{multline*}
(\ell-j)(\ell + 1 - j)(3 - 2j) - (j-1)j(2j + 5) + (1+2j)(24\beta j^2 + 24\beta j +3\ell) \\
\simeq (\ell - j)(-1 -j)(3 - 2j) - (j-1)j(2j+5) +(1 +2j)(24\beta j^2 + 24\beta j + 6) \\
\simeq (2 +j)(1 +j)(3 -2j) - (j-1)j(2j +5) + (1 + 2j)(24\beta j^2 + 24\beta j +6) \\
= (j^2 + 3j +2)(3 -2j) -j(2j^2 +3j -5) +(48\beta j^3 +72\beta j^2 + (24\beta  + 12)j + 6)\\
= -2j^3 -3j^2 +5j +6 -2j^3 -3j^2 +5j + (48\beta j^3 +72\beta j^2 + (24\beta  + 12)j + 6) \\
= (48\beta - 4)j^3 + (72\beta  - 6)j^2 + (24\beta  - 2)j
\end{multline*}
This polynomial must be zero modulo $\ell + 2$ so we must also have
$\beta = 1/12$.
\end{proof}

\begin{cor} \label{cor:e8} 
For $\mg{g} = \mg{sl}_2$ and $\l \in P_{\ell}$ even
\begin{equation}
\deg(\mb{V}_{\l}) = - \Bigl ( \frac{\l^2 - 3\l\ell + 2\ell^2 - \l + 2\ell}{24} \Bigr ) \ .
\end{equation}
\end{cor}
\begin{proof}
  This follows easily from the theorem using the computations made in
  its proof.
\end{proof}
One can see from the formula that for $\mg{sl}_2$, $\deg(\mb{V}_{\l})$
is always negative if $\ell > 0$. Based on many other computations, it
seems that this holds for all Lie algebras $\mg{g}$ (whenever $r_{\l}
> 0$) except for the following example.

\begin{cor}
  For all $g,n$ with $g > 0$ and $n \geq 1$ if $g=0$, $\mg{g}=
  \mg{e}_8$, $\ell =1$ and $\ov{0} = (0,\dots,0)$, we have
  $\mb{V}_{g,n,\ov{0}} \cong \mr{L}^{\otimes 4}$, where $\mr{L}$
  denotes the Hodge line bundle on $\mgnb$.
\end{cor}
\begin{proof}
  It is well known, and since $P_1 = \{0\}$ follows easily from the
  factorisation formula, that $r_{\ov{0}} =1$ for all $g,n$ as above.
  It thus suffices to compute the degree on all F-curves. If we have
  an F-curve of type $\mfb$ then again using $P_1 = \{0\}$ the degree
  is easily seen to be zero.

  Now consider the case $(g,n) = (1,1)$ and apply Theorem
  \ref{thm:m11}.  We have $\dim(\mg{e}_8) = 248$ and $h^{\vee} = 30$
  so it follows that in this case $\deg(\mb{V}_0) = 1/3$. Since
  $\deg(\mr{L}) = 1/12$ and $\pic(\mel)$ is a free abelian group
  generated by $\mr{L}$, the claim follows in this case.

  For general $g,n$ it follows from the propagation of vacuum that it
  suffices to consider the cases $(g,n) = (1,1)$ or $(g,n) = (g,0)$,
  $g>1$.  It is also well known that $\pic(\mgnb)$ is torsion
  free. Since $\mb{V}_{g,n,\ov{0}}$ is trivial on all F-curves of type
  $\mfb$ it follows that $\mb{V}_{g,n,\ov{0}}$ is a power of
  $\mr{L}$. To compute the precise power if $g > 1$, one reduces to
  the $(1,1)$ case by restricting to a family of stable curves of
  arithemtic genus $g$ obtained by gluing a fixed smooth curve of
  genus $g-1$ to a varying family of $1$-pointed curves of genus $1$
  and then using the factorisation formula.
\end{proof}

\begin{rem}
  The construction of the bundles of conformal blocks provides a
  natural section of $\mb{V}_{g,n,\ov{0}}$. For $\mg{g} = \mg{e}_8$
  and $g > 0$, using some observations of Faltings
  \cite[p.~12]{faltings-theta} this can be shown to be equal, upto a
  scalar, to the pullback via the Torelli map of the theta series
  associated to the $E_8$ lattice (which is a section of the fourth
  tensor power of the Hodge line bundle over $\mr{A}_g$, the moduli
  stack of principally polarised abelian vareties).
\end{rem}

\subsection{Nef bundles in higher genus}

As we have noted, the bundles of conformal blocks are often not nef on
$\mgnb$ if $g>0$. However, there is a canonical way of associating a
nef divisor class on $\mgnb$ to any $\ovl \in P_{\ell}^n$.
\begin{prop}
  For any $g \geq 0$ and $\ovl \in P_{\ell}^n$, there is a unique $t
  \in \Q$ such that $c_1(\mb{V}_{g,n,\ovl}) + tc_1(\mr{L})$ is nef on
  $\mgnb$ and is trivial on any family of elliptic tails.
\end{prop}

\begin{proof}
  By the functoriality of conformal blocks, the factorisation formula
  and Lemma \ref{lem:gen} it follows that $\mb{V}_{g,n,\ovl}$ has
  non-negative degree on any F-curve of $\mof$ type. It also has
  constant degree on any family of elliptic tails since any such
  family gives the same class in $N_1(\mgnb)$.  On the other hand
  $\mr{L}$ has degree zero on all F-curves of $\mof$ type and has a
  (constant) positive degree on any family of elliptic tails. It
  follows that there is a unique $t \in \Q$ (depending on $\ovl$ and
  $g$) such that $c_1(\mb{V}_{g,n,\ovl}) + tc_1(\mr{L})$ is F-nef and
  trivial on any family of elliptic tails.

  Using the functoriality of conformal blocks, the factorisation
  formula and Lemma \ref{lem:gen} again, it follows that the
  restriction of $\mb{V}_{g,n,\ovl}$ to the locus of flag curves is
  generated by sections, so has nef first Chern class. Since $\mr{L}$
  restricts to a trivial bundle on the locus of flag curves it follows
  from \cite[(0.3)]{GKM} that $c_1(\mb{V}_{g,n,\ovl}) + tc_1(\mr{L})$
  is nef on $\mgnb$.
\end{proof}

\section{Questions} \label{sec:questions} 

We conclude this paper by discussing some natural questions
concerning the bundles of conformal blocks.

\begin{ques}
  Given a simple Lie algebra $\mg{g}$ and an integer $n \geq 4$, is
  the closure of the subcone of $N^1(\mnb)$ generated by determinants
  of conformal blocks for $\mg{g}$ and all levels $\ell$ finitely
  generated?  If so, is there an algorithm for computing this cone?
\end{ques}

We do not know what to expect. For $n=5$ the determinants of conformal
blocks for $\mg{sl}_2$ generate the nef cone. This does not appear to
hold for $n=6$, in which case computer calculations\footnote{These
  were carried out in part using \tt{polymake} \cite{polymake}.}
suggest that the cone generated by conformal blocks for $\mg{sl}_2$
has $128$ vertices, $127$ coming from the critical level and the
remaining vertex corresponding to $\ovl = (1,1,1,1,1,1)$ and level
$1$.

\begin{ques}
  Given an integer $n$, do the determinants of conformal blocks for
  all simple Lie algebras and all levels $\ell$ generate the nef cone
  of $\mnb$?
\end{ques}

We do not know if this is true for any $n \geq 6$. However, for $n=6$
the cone generated by conformal blocks for both $\mg{sl}_2$ and
$\mg{sl}_3$ appears to strictly contain the cone for only $\mg{sl}_2$.

\smallskip

One may ask a similar question for $\mnb/S_n$.

\bigskip

For $g > 0$, the procedure we described at the beginning of \S
\ref{sec:g>0} shows that the Chern classes in cohomology lie in the
subring generated by the classes of divisors. We therefore ask the
following:
\begin{ques}
  Do the Chern classes in $\mathrm{\mr{CH}}^*(\mgnb)_{\Q}$ of
  all conformal blocks bundles lie in the subring generated by
  divisors?  If not, do they all lie in the tautological subring?
\end{ques}

The second part of the question was suggested by Brendan Hassett. It
seems likely that it has a positive answer.


\begin{thebibliography}{10}

\bibitem{abe-strange}
{\sc T.~Abe}, {\em Strange duality for parabolic symplectic bundles on a
  pointed projective line}, Int. Math. Res. Not. IMRN,  (2008), pp.~Art. ID
  rnn121, 47.

\bibitem{alexeev-swinarski}
{\sc V.~Alexeev and D.~Swinarski}, {\em Nef divisors on $\ov{M}_{0,n}$ from
  {GIT}}.
\newblock arXiv.0812.0778.

\bibitem{bourbaki-roots}
{\sc N.~Bourbaki}, {\em \'{E}l\'ements de math\'ematique. {F}asc. {XXXIV}.
  {G}roupes et alg\`ebres de {L}ie. {C}hapitre {IV}: {G}roupes de {C}oxeter et
  syst\`emes de {T}its. {C}hapitre {V}: {G}roupes engendr\'es par des
  r\'eflexions. {C}hapitre {VI}: syst\`emes de racines}, Actualit\'es
  Scientifiques et Industrielles, No. 1337, Hermann, Paris, 1968.

\bibitem{EV-logderham}
{\sc H.~Esnault and E.~Viehweg}, {\em Logarithmic de {R}ham complexes and
  vanishing theorems}, Invent. Math., 86 (1986), pp.~161--194.

\bibitem{faltings-theta}
{\sc G.~Faltings}, {\em Theta divisors on moduli spaces of bundles}.
\newblock http://www.dmv2006.uni-bonn.de/vortraege/faltings.pdf.

\bibitem{frenkel-benzvi}
{\sc E.~Frenkel and D.~Ben-Zvi}, {\em Vertex algebras and algebraic curves},
  vol.~88 of Mathematical Surveys and Monographs, American Mathematical
  Society, Providence, RI, second~ed., 2004.

\bibitem{GAP4}
{\sc The GAP~Group}, {\em {GAP -- Groups, Algorithms, and Programming, Version
  4.4.12}}, 2008.

\bibitem{polymake}
{\sc E.~Gawrilow and M.~Joswig}, {\em polymake: a framework for analyzing
  convex polytopes}, in Polytopes --- Combinatorics and Computation, G.~Kalai
  and G.~M. Ziegler, eds., Birkh\"auser, 2000, pp.~43--74.

\bibitem{giansiracusa-conformal}
{\sc N.~Giansiracusa}, {\em Conformal blocks and rational normal curves}.
\newblock arXiv:1012.4835.

\bibitem{GKM}
{\sc A.~Gibney, S.~Keel, and I.~Morrison}, {\em Towards the ample cone of
  {$\overline M\sb {g,n}$}}, J. Amer. Math. Soc., 15 (2002), pp.~273--294
  (electronic).

\bibitem{hassett-weighted}
{\sc B.~Hassett}, {\em Moduli spaces of weighted pointed stable curves}, Adv.
  Math., 173 (2003), pp.~316--352.

\bibitem{kapranov-chow}
{\sc M.~M. Kapranov}, {\em Chow quotients of {G}rassmannians. {I}}, in I. {M}.
  {G}el\cprime fand {S}eminar, vol.~16 of Adv. Soviet Math., Amer. Math. Soc.,
  Providence, RI, 1993, pp.~29--110.

\bibitem{keel-m0nbar}
{\sc S.~Keel}, {\em Intersection theory of moduli space of stable {$n$}-pointed
  curves of genus zero}, Trans. Amer. Math. Soc., 330 (1992), pp.~545--574.

\bibitem{keel-mckernan}
{\sc S.~Keel and J.~McKernan}, {\em Contractible extremal rays on
  {$\overline{M}_{0,n}$}}.
\newblock alg-geom/9607009.

\bibitem{littelmann-characters}
{\sc P.~Littelmann}, {\em Characters of representations and paths in {$\frak
  H^\ast_{\bold R}$}}, in Representation theory and automorphic forms
  ({E}dinburgh, 1996), vol.~61 of Proc. Sympos. Pure Math., Amer. Math. Soc.,
  Providence, RI, 1997, pp.~29--49.

\bibitem{looijenga-conformal}
{\sc E.~Looijenga}, {\em Conformal blocks revisited}.
\newblock math.{A}{G}/0507086.

\bibitem{mumford-theta1}
{\sc D.~Mumford}, {\em Tata lectures on theta. {I}}, vol.~28 of Progress in
  Mathematics, Birkh\"auser Boston Inc., Boston, MA, 1983.
\newblock With the assistance of C. Musili, M. Nori, E. Previato and M.
  Stillman.

\bibitem{silverman-aec}
{\sc J.~H. Silverman}, {\em The arithmetic of elliptic curves}, vol.~106 of
  Graduate Texts in Mathematics, Springer-Verlag, New York, 1986.

\bibitem{sorger-verlinde}
{\sc C.~Sorger}, {\em La formule de {V}erlinde}, Ast\'erisque,  (1996),
  pp.~Exp.\ No.\ 794, 3, 87--114.
\newblock S{\'e}minaire Bourbaki, Vol. 1994/95.

\bibitem{tsuchimoto}
{\sc Y.~Tsuchimoto}, {\em On the coordinate-free description of the conformal
  blocks}, J. Math. Kyoto Univ., 33 (1993), pp.~29--49.

\bibitem{tsuchiya-kanie}
{\sc A.~Tsuchiya and Y.~Kanie}, {\em Vertex operators in conformal field theory
  on {${\bf P}\sp 1$} and monodromy representations of braid group}, in
  Conformal field theory and solvable lattice models ({K}yoto, 1986), vol.~16
  of Adv. Stud. Pure Math., Academic Press, Boston, MA, 1988, pp.~297--372.

\bibitem{TUY}
{\sc A.~Tsuchiya, K.~Ueno, and Y.~Yamada}, {\em Conformal field theory on
  universal family of stable curves with gauge symmetries}, in Integrable
  systems in quantum field theory and statistical mechanics, vol.~19 of Adv.
  Stud. Pure Math., Academic Press, 1989, pp.~459--566.

\bibitem{ueno-conformal}
{\sc K.~Ueno}, {\em Introduction to conformal field theory with gauge
  symmetries}, in Geometry and physics ({A}arhus, 1995), vol.~184 of Lecture
  Notes in Pure and Appl. Math., Dekker, 1997, pp.~603--745.

\end{thebibliography}

\def\cprime{$'$}

\end{document}